\documentclass[8pt,a4paper,oneside,english,english]{article}
\usepackage[latin1]{inputenc}
\usepackage{indentfirst}

\usepackage{amsmath,amsfonts,amssymb,amsthm,amscd}
\RequirePackage{ifthen}
\RequirePackage{hyperref}

\usepackage{latexsym}
\usepackage{mathrsfs}
\usepackage{algorithm}
\usepackage[noend]{algorithmic}
\usepackage{paralist}
\usepackage{graphicx}
\usepackage{booktabs}
\usepackage{mathtools}
\usepackage{rotating}
\usepackage{tikz}
\usepackage{tikz-cd}

\usepackage{xcolor}

\usepackage{subcaption}

\usepackage[normalem]{ulem}
\usepackage[color,matrix,arrow]{xy}
\usepackage{mathtools}

\usetikzlibrary{patterns}

\usepackage{comment}
\usepackage{fontawesome5}

\usepackage[a4paper,margin=1in]{geometry}

\input xy
\xyoption{all}
\theoremstyle{plain}

\theoremstyle
{plain}
\newtheorem{theorem}{Theorem}[section]

\newtheorem{proposition}[theorem]{Proposition}

\newtheorem{fact}[theorem]{Fact}
\newtheorem{lemma}[theorem]{Lemma}
\newtheorem{corollary}[theorem]{Corollary}

\theoremstyle{definition}
\newtheorem{definition}[theorem]{Definition}

\newtheorem{example}[theorem]{Example}
\newtheorem{remark}[theorem]{Remark}

\newcommand{\leqdr}{\mathbin{\rotatebox[origin=c]{-45}{$\leq$}}}
\newcommand{\lequr}{\mathbin{\rotatebox[origin=c]{45}{$\leq$}}}
\newdir{d}{\leq}
\newdir{d}{\geq}
\newdir{dr}{\leqdr}
\newdir{ur}{\lequr}

\newcommand{\N}{\mathbb{N}}
\newcommand{\Z}{\mathbb{Z}}

\newcommand{\R}{\mathbb{R}}

\newcommand{\Tau}{\mathcal T}
\newcommand{\setdef}[2]{\left\{{#1} \ \middle|\ {#2}\right\}}

\newcommand{\GH}{d_{GH}}

\newcommand{\dint}{d_I}

\newcommand{\MP}[1]{(A_{#1},\chi_{#1})}
\newcommand{\extMP}[1]{(A_{#1},\chi_{#1}\colon X_{#1}\to\N)}

\newcommand{\newinf}{\mathop{\mathrm{inf}\vphantom{\mathrm{sup}}}}

\newcommand{\intmod}[1]{\mathcal{I}_{#1}}
\newcommand{\sixPackPersistenceModule}[6]{{#6}_{#5}^{{#3}\leq{#4}}({#2};{#1})}
\newcommand{\Dom}[5]{\sixPackPersistenceModule{#1}{#2}{#3}{#4}{#5}{\mathrm{Dom}}}
\newcommand{\Cod}[5]{\sixPackPersistenceModule{#1}{#2}{#3}{#4}{#5}{\mathrm{Cod}}}
\newcommand{\Img}[5]{\sixPackPersistenceModule{#1}{#2}{#3}{#4}{#5}{\mathrm{Im}}}
\newcommand{\Ker}[5]{\sixPackPersistenceModule{#1}{#2}{#3}{#4}{#5}{\mathrm{Ker}}}
\newcommand{\Cok}[5]{\sixPackPersistenceModule{#1}{#2}{#3}{#4}{#5}{\mathrm{Cok}}}
\newcommand{\Rel}[5]{\sixPackPersistenceModule{#1}{#2}{#3}{#4}{#5}{\mathrm{Rel}}}

\DeclareMathOperator{\im}{im}
\DeclareMathOperator{\dom}{dom}
\DeclareMathOperator{\codom}{cod}
\DeclareMathOperator{\coker}{coker}

\DeclareMathOperator{\diam}{diam}

\DeclareMathOperator{\dist}{dist}
\DeclareMathOperator{\dis}{dis}
\DeclareMathOperator{\codis}{codis}
\DeclareMathOperator{\gr}{Gr}
\DeclareMathOperator{\Dgm}{Dgm}

\DeclareMathOperator{\ecc}{ecc}
\DeclareMathOperator{\rad}{rad}
\DeclareMathOperator{\sep}{sep}

 \author{Ond\v{r}ej Draganov, Sophie Rosenmeier, Nicol\`o Zava}
\title{Gromov-Hausdorff distance between chromatic metric pairs and stability of the six-pack}
\date{}

\begin{document}

\maketitle

\begin{abstract}
Chromatic metric pairs consist of a metric space and a coloring function partitioning a subset thereof into various colors. 
It is a natural extension of the notion of chromatic point sets studied in chromatic topological data analysis.
A useful tool in the field is the six-pack, a collection of six persistence diagrams, summarizing homological information about how the colored subsets interact.
We introduce a suitable generalization of the Gromov-Hausdorff distance to compare chromatic metric pairs. We show some basic properties and validate this definition by obtaining the stability of the six-pack with respect to that distance.
We conclude by discussing its restriction to metric pairs and its role in the stability of the \v{C}ech persistence diagrams.
\end{abstract}


\tableofcontents

\section{Introduction}

In topological data analysis (TDA), computable topological invariants are constructed to extract patterns and meaningful features from datasets.
For practical use, we want the invariants to be stable under perturbations: 
if two datasets are similar, differing only by some small errors acquired in the data collection, then their associated invariants should be similar as well.
To rigorously formulate such a \emph{stability result}, we need to choose a notion of distance between the invariants on one hand, and between the data sets on the other.  
The cornerstone example of topological invariants in TDA are persistence diagrams summarizing persistent homology of the studied spaces.
The first stability result for persistence diagrams appeared in \cite{Coh-SteEdeHar} using the supremum metric on filter functions to compare data sets, and bottleneck distance to compare the persistence diagrams.

Data can often be represented as metric spaces. One classical way to compare metric spaces is the Gromov-Hausdorff distance---a dissimilarity notion introduced
by Gromov \cite{gromov:81} to study the convergence of metric structures (see Edwards \cite{edwards:75} and Kadets \cite{kadets:75} for earlier notions). In the past decades, it has found applications to dataset and shape matching and recognition, pioneered by M\'emoli and Sapiro \cite{MemSap}.
Its original definition relies on the classic notion of Hausdorff distance, which, despite being intuitive, strongly depends on the particular embedding of the objects we are comparing into a common space. 
The Gromov-Hausdorff distance removes that dependency by considering the ``best'' possible isometric embedding to a common space---one that minimizes the Hausdorff distance. 
There are characterizations removing the need for an explicit third common space; one using distortions of correspondences, and one using distortions and codistortions of pairs of maps (see, for example, \cite{BurBurIva,KalOst,Tuz}).

In \cite{ChaCoh-SteGuiMemOud} (see  
also \cite{ChaDeSOud}), the stability for Vietoris-Rips persistence diagrams using the Gromov-Hausdorff distance was proven: 
for every pair of metric spaces, 
the bottleneck distance of their Vietoris-Rips persistence diagrams is at most twice their Gromov-Hausdorff distance. 
This also implies the previous stability result discussed in \cite{Coh-SteEdeHar} in the case of Vietoris-Rips persistence diagrams.

Data often comes in modalities that are not best represented simply by a metric space---e.g., a directed network is given instead of distances, or some additional information is carried by data points that is not easily encoded by distances. The Gromov-Hausdorff distance can be adapted to fit such situations, and this approach has been used to prove the stability of invariants defined on those more general spaces.
We mention the stability of Vietoris-Rips and Dowker persistence diagrams of general networks using the network distance \cite{ChoMem3} and the stability of the grounded persistence path homology shown in \cite{ChaHarTil}. As for further generalizations of the Gromov-Hausdorff distance, we also mention \cite{directed,Khe,Zav_q}.

One natural extension of studying a single metric space is studying the interaction of two or more. In this paper, we consider an ambient metric space with disjoint subspaces, and topological invariants capturing their mutual interaction. The spatial interaction of several finite sets of points in Euclidean space is relevant in different fields. Such data is modeled as \emph{chromatic point sets}: a finite subset $X
\subseteq \R^n$ colored by a map $\chi\colon X
\rightarrow \N$. 
Examples of chromatic point sets are positions of biological cells in a tumor microenvironment colored by their type---e.g., cancer cells versus tumor cells---or positions of atoms in a material colored by their element. Different approaches have been developed to study such situations using persistent homology. For example, witness complexes \cite{DabMemFraCar} and Dowker complexes \cite{ChoMem3} were used in \cite{StoDheBul} to distinguish different simulated tumor microenvironments. One downside of this approach is the lack of stability of the complexes against perturbations in the positions of the points.

In \cite{CulDraEdeSag1} (see also \cite{CulDraEdeSag2}) a different approach was proposed consisting in summarizing the interactions using \emph{six-packs} of persistence diagrams---a work building on a discretization in \cite{BobRea} and results from \cite{Coh-SteEdeHarMor} on the computation of image, kernel and cokernel persistence, which also proves stability of the used diagrams with respect to perturbations of the points and the bottleneck distance for the diagrams.

Given an inclusion of two filtered topological spaces, a six-pack is a collection of six persistence diagrams. Namely, those of the domain, codomain and the quotient of the codomain by the domain, and those of the images, kernels and cokernels of the homology maps induced by the inclusion.
For more details, see \cite[Section~5.2]{CulDraEdeSag1}.

In this paper, we define the notion of {\em chromatic metric pair} generalizing the notion of chromatic point sets. It consists of a {\em metric pair} $(A,X)$, where $X$ is a subspace of an ambient metric space $A$, and a coloring function $\chi\colon X\to\N$.  We consider morphisms between two chromatic metric pairs depending on a constraint set $C\subseteq\mathcal P(\N)$, which prescribes where points with certain colors can be mapped. Using this notion and inspired by the characterization of the Gromov-Hausdorff distance using pairs of maps (\cite{KalOst}), we introduce and study the {\em $C$-constrained Gromov-Hausdorff distance} between chromatic metric pairs. Several properties of the usual Gromov-Hausdorff distance can be generalized to this setting. In particular, we provide the counterparts of the usual characterizations using either embedding into a common space or correspondences, and show that, under compactness hypotheses, the infimum involved in the definition can be achieved. Finally, we validate our definition by proving the stability of six-packs of persistence diagrams. As a particular case of our result, the stability of the persistence diagrams induced by the ambient \v{C}ech complex follows. We also provide an alternative proof of this result relying on notions introduced in \cite{Mem_tripods}.

\medskip

\noindent {\bf Structure of the paper.} In Section \ref{sec:background}, we provide the needed background concerning the usual notion of Gromov-Hausdorff distance between metric spaces. Section \ref{sec:chromatic_metric_pairs} is devoted to the introduction of chromatic metric pairs and $C$-constrained maps between them. In particular, in \S\ref{sub:constraints_vs_topologies}, we investigate the relationship between constraint sets and Alexandrov topologies of $\N$. We introduce and study the $C$-constrained Gromov-Hausdorff distance in Section \ref{sec:C-constrained_GH}. We describe suitable invariants and provide explicit computations in \S\ref{sub:invariants}, prove characterizations of the new notion in \S\ref{sub:characterisations} and use them to discuss related concepts (\S\ref{subsub:related_notions}), and show the existence of distortion and codistortion-minimizing pair of maps in \S\ref{sub:inf_is_min}. Finally, Section \ref{sec:stability} is devoted to recalling some background from persistent homology, the construction of the six-pack (\S\ref{sub:ph-background}) and proving its stability (\S\ref{sub:stability}). We discuss the stability of the ambient \v{C}ech persistence diagrams in~\S\ref{sub:Cech_stability}.

\medskip
\noindent {\bf Notation.} We denote by $\N$ the set of natural numbers including $0$. For every set $X$, $\mathcal P(X)$ denotes its power set.

\medskip
\noindent \textbf{Disclaimer.} Closely related work is independently studied by Yaoying Fu, Lander Ver Hoef, Evgeniya Lagoda, Shiying Li, Tom Needham and Morgan Weiler with a preprint titled ``Persistent Homology for Labeled Datasets: Stability and Generalized Landscapes'' planned to be released soon. Later versions of the papers will explain the relation and differences in more detail.

\medskip
\noindent {\bf Acknowledgements.}

The research was supported by the FWF Grant, Project number I4245-N35 and AI4scMed, France 2030 ANR-22-PESN-000 from Agence Nationale de la Recherche. We would like to thank Ulrich Bauer for bringing the paper \cite{Mem_tripods} to our attention, and Thomas Weighill for the discussion about the relationship between constraint sets with topologies on $\N$ that ultimately led to Subsection~\ref{sub:constraints_vs_topologies}.

\section{Gromov-Hausdorff distance and classic stability results}\label{sec:background}

In this section we present background about the Gromov-Hausdorff distance. For more details, we refer the reader to \cite{BurBurIva,Tuz,Pet}. 
We first recall three equivalent definitions of the Gromov-Hausdorff distance, one using distortions of correspondences, one using distortions and codistortions of pairs of maps and one via embeddings into a common space.

Given a relation $R\subseteq X\times Y$ between two metric spaces, $(X, d_X)$ and $(Y, d_Y)$, its {\em distortion} is \[
    \dis R =
    \sup_{(x,y), (x',y')\in {R}}
    \lvert d_X(x,x')-d_Y(y,y')\rvert.
\]
A relation $R$ is a {\em correspondence} if every point in both $X$ and $Y$ is related, that is, for every $x\in X$ and $y\in Y$, there are $y_x\in Y$ and $x_y\in X$ such that $(x,y_x),(x_y,y)\in R$. Given two sets, $X$ and $Y$, we denote by $\mathcal R(X,Y)$ the set of all correspondences between $X$ and $Y$.

\begin{definition}\label{def:GH-distance}
    Let $X$ and $Y$ be two metric spaces. Their {\em Gromov-Hausdorff distance} is defined as \[
        \GH(X,Y)=\frac{1}{2}
        \inf_{R\in\mathcal R(X,Y)}
        \dis R.
    \]
\end{definition}

The intuition is that we match points of $X$ and $Y$ together, allowing using each point multiple times, in such a way that distances on one side of the relation edges are closely approximated by the distances on the other side of the edges.

The Gromov-Hausdroff distance is clearly symmetric and it satisfies the triangle inequality. It is also true that two isometric metric spaces $X$ and $Y$ satisfy $\GH(X,Y)=0$. The converse implication is not generally true---there are non-isometric metric spaces with zero Gromov-Hausdorff distance, e.g., the open and the closed unit intervals in $\R$. However, under compactness assumptions, the implication holds (see Corollary \ref{coro:0-GH}). 

We next give a characterization of the Gromov-Hausdorff distance due to Kalton and Ostrovskii (\cite{KalOst}) using pairs of maps instead of correspondences. Given a map $f\colon X\to Y$, we denote by $\gr f$ its {\em graph}, i.e., $\gr f=\{(x,f(x))\mid x\in X\}\subseteq X\times Y$. The {\em distortion of $f$}, $\dis f$, is defined as the distortion of its graph. The \emph{codistortion} of two maps $f\colon X\to Y$ and $g\colon Y\to X$ between metric spaces is defined as \[
    \codis(f,g)=
    \sup_{\substack{x\in X \\ \,y\in Y}}
    \lvert d_X(x,g(y))-d_Y(f(x),y)\rvert.
\]

\begin{theorem}\label{thm:Kalton_Ostrovskii}
    Given two metric spaces $X$ and $Y$,
    \begin{equation}\label{eq:Kalton_Ostrovskii}
        \GH(X,Y)=\frac{1}{2}\inf_{\substack{f\colon X\to Y\\g\colon Y\to X}}\max\{\dis f,\dis g,\codis(f,g)\}.
    \end{equation}
\end{theorem}
For completeness, we recall the simple proof.
\begin{proof}
    Let $R\subseteq X\times Y$ be a correspondence. We define two maps $f\colon X\to Y$ and $g\colon Y\to X$ as follows. Since $R$ is a correspondence, for every $x\in X$, we can pick a point $y_x\in Y$ such that $(x,y_x)\in R$ and set $f(x)=y_x$, and, similarly, $g$ can be defined. Then, $\dis f\leq\dis R$, $\dis g\leq\dis R$ and $\codis(f,g)\leq\dis R$. Hence, the inequality $(\geq)$ in \eqref{eq:Kalton_Ostrovskii} follows.
    
    Consider now two maps $f\colon X\to Y$ and $g\colon Y\to X$. Define the correspondence $R=\gr f\cup(\gr g)^{-1}$, where, for a relation $S\subseteq Y\times X$, $S^{-1}=\{(x,y)\mid (y,x)\in S\}$. The distortion of $R$ is precisely $\max\{\dis f,\dis g,\codis(f,g)\}$, and so the inequality $(\leq)$ in \eqref{eq:Kalton_Ostrovskii} is shown.
\end{proof}

Let us conclude with a third characterization---Gromov's original definition. Let $(X,d)$ be a metric space. The {\em Hausdorff distance} between two subsets $A,B\subseteq X$ is defined as
\[
    d_H(A,B)=\newinf_{R\in\mathcal R(A,B)}\sup_{(a,b)\in R}d(a,b),
\]
intuitively, the maximum edge length in the best possible correspondence between $A$ and $B$.

Equivalently, this is the smallest $\varepsilon$ such that $\varepsilon$-thickening of $A$ contains $B$ and $\varepsilon$-thickening of $B$ contains $A$ (see Figure \ref{fig:Hausdorff}). Indeed, this is clearly the case if all edges of a correspondence have length at most $\varepsilon$, and on the other hand we consider the correspondence connecting each point of one set to the nearest neighbor(s) from the other set.

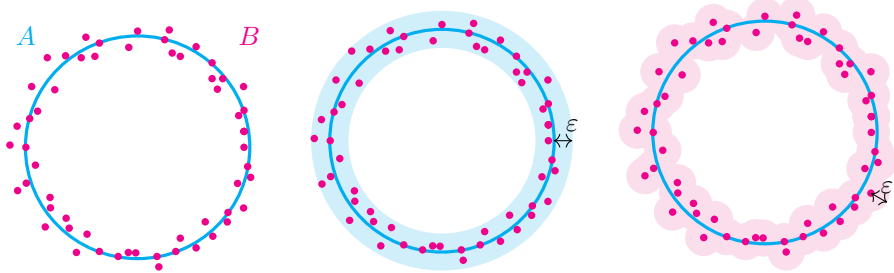
\begin{figure}[h!]
    \centering
    
	\begin{subfigure}{0.25\textwidth}
		\centering
		\begin{tikzpicture}[scale=0.7]
		\tikzstyle{blueshade} = [fill=cyan!20,fill opacity=0.4]
		\tikzstyle{redshade} = [fill=magenta!20,fill opacity=0.4]	
		
		\draw[very thick,cyan] (0,0) circle (60pt);
		\draw[cyan] (-2.1,2.1) node{$A$};
		\draw[magenta] (2.1,2.1) node{$B$}; 
		\fill[magenta] (30:2.3) circle (2pt) (60:2.2) circle (2pt) (65:1.9) circle (2pt) (70:1.9) circle (2pt) (75:2.3) circle (2pt) (76:2.1) circle (2pt) (90:2.2) circle (2pt) (95:1.9) circle (2pt) (110:2.1) circle (2pt) (115:2.3) circle (2pt) (115:1.9) circle (2pt) (122:2) circle (2pt) (128:2.2) circle (2pt) (135:2.4) circle (2pt) (145:1.9) circle (2pt) (150:2.3) circle (2pt) (160:2) circle (2pt) (165:2.1) circle (2pt) (170:2.3) circle (2pt) (180:2.1) circle (2pt) (179:2.4) circle (2pt) (190:1.95) circle (2pt) (198:2.2) circle (2pt) (200:2.4) circle (2pt) (210:1.9) circle (2pt) (215:2) circle (2pt) (221:2.3) circle (2pt) (225:1.9) circle (2pt) (230:2) circle (2pt) (240:1.9) circle (2pt) (240:2.3) circle (2pt) (250:2.1) circle (2pt) (260:2.1) circle (2pt) (265:2) circle (2pt) (269:2) circle (2pt) (280:2.1) circle (2pt) (280:2.3) circle (2pt) (290:2.1) circle (2pt) (295:1.9) circle (2pt) (300:2.2) circle (2pt) (310:2.2) circle (2pt) (312:1.9) circle (2pt) (320:2.2) circle (2pt) (324:2.1) circle (2pt) (330:2.3) circle (2pt) (340:2) circle (2pt) (345:2.2) circle (2pt) (350:2.1) circle (2pt)
		(2,0.3) circle (2pt) (2,0) circle (2pt) (1.9,0.6) circle (2pt) (2,0.7) circle (2pt) (1.5,1.1) circle (2pt) (1.6,1.3) circle (2pt) (1.4,1.3) circle (2pt) (1.4,1.6) circle (2pt) (2,0.3) circle (2pt);
	\end{tikzpicture}
\end{subfigure}
	\begin{subfigure}{0.25\textwidth}
		\centering
		\begin{tikzpicture}[scale=0.7]
			\tikzstyle{blueshade} = [fill=cyan!40,fill opacity=0.4]
			\tikzstyle{redshade} = [fill=magenta!20,fill opacity=0.4]	
			\fill [blueshade,even odd rule] (0,0) circle[radius=50pt] circle[radius=70pt];
			\draw[very thick,cyan] (0,0) circle (60pt);
			\fill[magenta] (30:2.3) circle (2pt) (60:2.2) circle (2pt) (65:1.9) circle (2pt) (70:1.9) circle (2pt) (75:2.3) circle (2pt) (76:2.1) circle (2pt) (90:2.2) circle (2pt) (95:1.9) circle (2pt) (110:2.1) circle (2pt) (115:2.3) circle (2pt) (115:1.9) circle (2pt) (122:2) circle (2pt) (128:2.2) circle (2pt) (135:2.4) circle (2pt) (145:1.9) circle (2pt) (150:2.3) circle (2pt) (160:2) circle (2pt) (165:2.1) circle (2pt) (170:2.3) circle (2pt) (180:2.1) circle (2pt) (179:2.4) circle (2pt) (190:1.95) circle (2pt) (198:2.2) circle (2pt) (200:2.4) circle (2pt) (210:1.9) circle (2pt) (215:2) circle (2pt) (221:2.3) circle (2pt) (225:1.9) circle (2pt) (230:2) circle (2pt) (240:1.9) circle (2pt) (240:2.3) circle (2pt) (250:2.1) circle (2pt) (260:2.1) circle (2pt) (265:2) circle (2pt) (269:2) circle (2pt) (280:2.1) circle (2pt) (280:2.3) circle (2pt) (290:2.1) circle (2pt) (295:1.9) circle (2pt) (300:2.2) circle (2pt) (310:2.2) circle (2pt) (312:1.9) circle (2pt) (320:2.2) circle (2pt) (324:2.1) circle (2pt) (330:2.3) circle (2pt) (340:2) circle (2pt) (345:2.2) circle (2pt) (350:2.1) circle (2pt)
			(2,0.3) circle (2pt) (2,0) circle (2pt) (1.9,0.6) circle (2pt) (2,0.7) circle (2pt) (1.5,1.1) circle (2pt) (1.6,1.3) circle (2pt) (1.4,1.3) circle (2pt) (1.4,1.6) circle (2pt) (2,0.3) circle (2pt);
			\draw[<->] (0:2.1)--(0:2.45) node[above]{$\varepsilon$};
		\end{tikzpicture}
	\end{subfigure}
\begin{subfigure}{0.25\textwidth}
		\centering
		\begin{tikzpicture}[scale=0.7]
			\tikzstyle{blueshade} = [fill=cyan!20,fill opacity=0.4]
			\tikzstyle{redshade} = [fill=magenta!40,fill opacity=0.4]	
			\fill[redshade] (2,0.3) circle (10pt) (2,0) circle (10pt) (1.9,0.6) circle (10pt) (2,0.7) circle (10pt) (1.5,1.1) circle (10pt) (1.6,1.3) circle (10pt) (1.4,1.3) circle (10pt) (1.4,1.6) circle (10pt) (2,0.3) circle (10pt) (2,0.3) circle (10pt) (30:2.3) circle (10pt) (60:2.2) circle (10pt) (65:1.9) circle (10pt) (70:1.9) circle (10pt) (75:2.3) circle (10pt) (76:2.1) circle (10pt) (90:2.2) circle (10pt) (95:1.9) circle (10pt) (110:2.1) circle (10pt) (115:2.3) circle (10pt) (115:1.9) circle (10pt) (122:2) circle (10pt) (128:2.2) circle (10pt) (135:2.4) circle (10pt) (145:1.9) circle (10pt) (150:2.3) circle (10pt) (160:2) circle (10pt) (165:2.1) circle (10pt) (170:2.3) circle (10pt) (180:2.1) circle (10pt) (179:2.4) circle (10pt) (190:1.95) circle (10pt) (198:2.2) circle (10pt) (200:2.4) circle (10pt) (210:1.9) circle (10pt) (215:2) circle (10pt) (221:2.3) circle (10pt) (225:1.9) circle (10pt) (230:2) circle (10pt) (240:1.9) circle (10pt) (240:2.3) circle (10pt) (250:2.1) circle (10pt) (260:2.1) circle (10pt) (265:2) circle (10pt) (269:2) circle (10pt) (280:2.1) circle (10pt) (280:2.3) circle (10pt) (290:2.1) circle (10pt) (295:1.9) circle (10pt) (300:2.2) circle (10pt) (310:2.2) circle (10pt) (312:1.9) circle (10pt) (320:2.2) circle (10pt) (324:2.1) circle (10pt) (330:2.3) circle (10pt) (340:2) circle (10pt) (345:2.2) circle (10pt) (350:2.1) circle (10pt);
			\draw[very thick,cyan] (0,0) circle (60pt);
			\fill[magenta] (30:2.3) circle (2pt) (60:2.2) circle (2pt) (65:1.9) circle (2pt) (70:1.9) circle (2pt) (75:2.3) circle (2pt) (76:2.1) circle (2pt) (90:2.2) circle (2pt) (95:1.9) circle (2pt) (110:2.1) circle (2pt) (115:2.3) circle (2pt) (115:1.9) circle (2pt) (122:2) circle (2pt) (128:2.2) circle (2pt) (135:2.4) circle (2pt) (145:1.9) circle (2pt) (150:2.3) circle (2pt) (160:2) circle (2pt) (165:2.1) circle (2pt) (170:2.3) circle (2pt) (180:2.1) circle (2pt) (179:2.4) circle (2pt) (190:1.95) circle (2pt) (198:2.2) circle (2pt) (200:2.4) circle (2pt) (210:1.9) circle (2pt) (215:2) circle (2pt) (221:2.3) circle (2pt) (225:1.9) circle (2pt) (230:2) circle (2pt) (240:1.9) circle (2pt) (240:2.3) circle (2pt) (250:2.1) circle (2pt) (260:2.1) circle (2pt) (265:2) circle (2pt) (269:2) circle (2pt) (280:2.1) circle (2pt) (280:2.3) circle (2pt) (290:2.1) circle (2pt) (295:1.9) circle (2pt) (300:2.2) circle (2pt) (310:2.2) circle (2pt) (312:1.9) circle (2pt) (320:2.2) circle (2pt) (324:2.1) circle (2pt) (330:2.3) circle (2pt) (340:2) circle (2pt) (345:2.2) circle (2pt) (350:2.1) circle (2pt)
			(2,0.3) circle (2pt) (2,0) circle (2pt) (1.9,0.6) circle (2pt) (2,0.7) circle (2pt) (1.5,1.1) circle (2pt) (1.6,1.3) circle (2pt) (1.4,1.3) circle (2pt) (1.4,1.6) circle (2pt) (2,0.3) circle (2pt);
			\draw[<->] (330:2.3)--(330:2.65) node[above]{$\varepsilon$};
		\end{tikzpicture}
	\end{subfigure}
	\caption{\footnotesize Hausdorff distance for two subsets of the plane: $A$, a circle, in \textit{cyan}, and $B$, a collection of dots, in \textit{magenta}. Since $A$ is contained in the union of balls centered in the points of $B$ with radius $\varepsilon$ (\textit{right}), and $B$ lies inside the union of balls centered in points of $A$ with the same radius (\textit{center}), their Hausdorff distance is at most $\varepsilon$.}\label{fig:Hausdorff}
\end{figure}

The Hausdorff distance provides another, more geometrical, characterization of the Gromov-Hausdorff distance. We recall the constructions used to prove the result below. For more details, we refer the interested reader to \cite{KalOst,ChoMem3}.
\begin{theorem}\label{thm:GH_via_embed}
    Let $(X,d_X)$ and $(Y,d_Y)$ be two metric spaces. Then, 
    \[
    \GH(X,Y)=\inf_{d\in\mathcal D(X,Y)}d_H(i_X(X),i_Y(Y)),
    \]
    where $i_X$ and $i_Y$ denote the canonical inclusions of $X$ and $Y$ into the disjoint union $X\sqcup Y$ and where $\mathcal D(X,Y)$ is the set of all {\em admissible metrics}, i.e., metrics $d$ on $X\sqcup Y$ such that $d|_{i_X(X)\times i_X(X)}=d_X$ and $d|_{i_Y(Y)\times i_Y(Y)}=d_Y$.
\end{theorem}
\begin{proof}
    Every correspondence $R\in\mathcal R(i_X(X),i_Y(Y))$ trivially induces a correspondence $R'$ between $X$ and $Y$. Furthermore, $\dis R'\leq 2\sup_{(x,y)\in R}d(i_X(x),i_Y(y))$ for any $d \in \mathcal D(X,Y)$. Therefore, the inequality $(\leq)$ holds.

    On the other hand, let $R$ be a correspondence between $X$ and $Y$ with distortion $2\varepsilon>0$. Construct an admissible metric $d_R$ on $X\sqcup Y$ as follows:
    \[
    d_R(z,z')=\begin{cases}
        \begin{aligned}
            &d_X(z,z') &\text{if $z,z'\in X$,}\\
            &d_Y(z,z') &\text{if $z,z'\in Y$,}\\
            &\inf_{(x,y)\in R}(d_X(z,x)+\varepsilon+d_Y(y,z')) &\text{if $z\in X$ and $z'\in Y$.}
        \end{aligned}
    \end{cases}
    \]
    Since $\dis R=2\varepsilon$, it can be checked that $d_R$ is a metric and, by construction, it is compatible.
    The~inequality $(\geq)$ follows as, clearly, $d_H(X,Y) = \varepsilon$.
\end{proof}

Let us state some known properties of the Gromov-Hausdorff distance.

\begin{proposition}[\cite{ChoMem_geo}]\label{prop:distortion_minimizing_correspondence}
    Given two compact metric spaces, $X$ and $Y$, there exists a correspondence $R\in\mathcal R(X,Y)$ such that $\GH(X,Y)=\frac{1}{2}\dis R.$
\end{proposition}

Let us recall that a map $f\colon(X,d_X)\to(Y,d_Y)$ between metric spaces is an {\em isometric embedding} if  $d_Y(f(x),f(x'))=d_X(x,x')$ for every $x,x'\in X$. If it is also bijective, then it is called an {\em isometry}. If such $f$ exists, $X$ and $Y$ are said to be {\em isometric}.

\begin{corollary}\label{coro:0-GH}
    
    Two compact metric spaces, $X$ and $Y$, are isometric if and only if $\GH(X,Y)=0$.
\end{corollary}
\begin{proof}
    If $X$ and $Y$ are isometric, then trivially $\GH(X,Y)=0$---the graph of the isometry is a correspondence with distortion $0$. 
    If $\GH(X,Y)=0$, then Proposition \ref{prop:distortion_minimizing_correspondence} implies the existence of a correspondence $R$ with $\dis R=0$. It is easy to check that $R$ itself must then be a graph of an isometry.
    
\end{proof}

In general, computing the Gromov-Hausdorff distance is an NP-hard problem, and so is approximating it by a factor of $3$ for finite trees with unit-length edges \cite{AgaFoxNatSidWan,Sch}.
It is, therefore, convenient to find various tractable stable invariants bounding it.
Below we recall a 
simple, but important example. For more, see, e.g., \cite{Mem12}.
\begin{proposition}\label{prop:GH_and_diam}
    Let $X$ and $Y$ be two metric spaces with finite diameter. Then,
    \[
    \frac{\lvert\diam X-\diam Y\rvert}{2}\leq \GH(X,Y)\leq \frac{\max\{\diam X,\diam Y\}}{2}.
    \]
\end{proposition}

\section{Chromatic metric pairs and $C$-constrained maps}\label{sec:chromatic_metric_pairs}

A {\em metric pair} is defined as a pair $(A,X)$ consisting of a metric space $A$ and a subset $X$ thereof. A~map $f\colon A_1\to A_2$ between metric spaces is a {\em pair map} between two metric pairs $(A_1,X_1)$ and $(A_2,X_2)$ if $f(X_1)\subseteq X_2$. It is a \emph{pair isomorphism} if it is an isometry and both $f$ and $f^{-1}$ are pair maps; equivalently, it is an isometry and $f(X_1) = X_2$. 
Inspired by the study of the Gromov-Hausdorff distance on pairs of metric spaces in \cite{GomChe}, in Section \ref{sec:C-constrained_GH}, we provide a similar notion for metric pairs with a coloring on the subspace. Although our definition gives a closely related notion of distance for metric pairs, we took a different approach.
Indeed, starting from Kalton and Ostrovkii's characterization of the usual Gromov-Hausdorff distance (Theorem \ref{thm:Kalton_Ostrovskii}), we restrict the family of maps we are allowed to use to minimize the distortion and codistortion.

The main object of study in this paper is akin to metric pairs, but instead of a single subspace, we consider several disjoint subspaces. We describe this as certain points having colors, which we denote by natural numbers.

\begin{definition}\label{def:chromatic-metric-pair}
    We define a \emph{chromatic metric pair} as a pair $\MP{}$, where $A$ is a metric space and $\chi\colon X\rightarrow \N$ a coloring of its 
    subspace $X\subseteq A$.
    When we need to specify the domain of $\chi$, we write $\extMP{}$. If $A=X$, the pair is a {\em chromatic metric space}, and we also denote it simply by $\chi$.
    We say that a point $x\in X$ is \emph{colored by} $\sigma\subseteq\N$, or \emph{$\sigma$-colored}, if $\chi(x)\in\sigma$.
\end{definition}

Given a chromatic metric pair $\extMP{}$, any subset $Y$ of $A$ forms a \emph{subspace} endowed with the chromatic metric pair structure $(Y,\chi|_{Y\cap X})$. See Figure~\ref{fig:cmp_and_subspace} for an example.

\begin{figure}[h!]
    \centering
    \begin{subfigure}{0.48\textwidth}
    \centering
        \begin{tikzpicture}[scale=1.2]
            \fill[lightgray] (0,0) circle (35pt);
            \fill[magenta] (0,0.5) circle (13pt);
            \fill[lightgray] (0,0.5) circle (10pt);
            \fill[cyan] (0,-0.5) circle (13pt);
            \fill[lightgray] (0,-0.5) circle (10pt);
            \draw (0,0) circle (35pt);
            \draw (0,0.5) circle (13pt);
            \draw (0,0.5) circle (10pt);
            \draw (0,-0.5) circle (13pt);
            \draw (0,-0.5) circle (10pt);            \draw (1.05,1.3) node{$X\xrightarrow{\chi}\N$};
            \draw (-1.6,-1.6)--(1.6,-1.6)--(1.6,1.6)--(-1.6,1.6)--(-1.6,-1.6);
            \draw (-1.7,1.3) node[left]{$A$};
            \draw[dashed] (-1.4,-1.4)--(0,-1.4)--(0,1.4)--(-1.4,1.4)--(-1.4,-1.4);
            \draw (-1.25,-1.25) node{$Y$};
        \end{tikzpicture}
    \end{subfigure}
    \begin{subfigure}{0.48\textwidth}
            \centering
            \begin{tikzpicture}[scale=1.2]
            \fill[lightgray] (0,0) circle (35pt);
            \fill[magenta] (0,0.5) circle (13pt);
            \fill[lightgray] (0,0.5) circle (10pt);
            \fill[cyan] (0,-0.5) circle (13pt);
            \fill[lightgray] (0,-0.5) circle (10pt);
            \draw (0,0) circle (35pt);
            \draw (0,0.5) circle (13pt);
            \draw (0,0.5) circle (10pt);
            \draw (0,-0.5) circle (13pt);
            \draw (0,-0.5) circle (10pt);            
            \draw[lightgray] (-1.6,-1.6)--(1.6,-1.6)--(1.6,1.6)--(-1.6,1.6)--(-1.6,-1.6);            
            \draw (-1.4,-1.4)--(0,-1.4)--(0,1.4)--(-1.4,1.4)--(-1.4,-1.4);
            \draw (-1.25,-1.25) node{$Y$};
            \fill[white,opacity=0.8] (-1.4,-1.6)--(-1.4,1.4)--(0,1.4)--(0,-1.4)--(-1.4,-1.4)--(-1.4,-1.6)--(1.6,-1.6)--(1.6,1.6)--(-1.6,1.6)--(-1.6,-1.6)--(-1.4,-1.6);
            \draw (1.1,1.3) node{$X\cap Y\xrightarrow{\chi|_{X\cap Y}}\N$};
        \end{tikzpicture}
    \end{subfigure}  
    \caption{\footnotesize On the left-hand side, an example of a chromatic metric pair $\MP{}$ and a subset $Y$ of $A$. On the right-hand side, we highlight the chromatic metric pair structure on $Y$ induced by $(A, \chi)$.}
    \label{fig:cmp_and_subspace}
\end{figure}
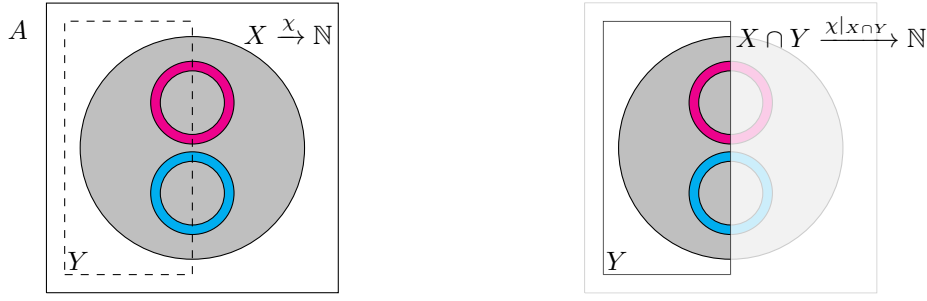

\begin{definition}
    Let $\sigma$ be a subset of $\N$. We say that a map $f\colon\MP{1}\to\MP{2}$ between chromatic metric pairs is {\em $\sigma$-constrained} if 
    \[
        f(\chi_1^{-1}(\sigma))\subseteq\chi_2^{-1}(\sigma),
    \]
    that is, if a point $x\in X_1$ is $\sigma$-colored, then also $f(x) \in X_2$ is $\sigma$-colored. For a family $C$ of subsets of $\N$, we say that $f$ is {\em $C$-constrained} if it is $\sigma$-constrained for every $\sigma\in C$.

    Given two constraint sets $\N\in C_1,C_2\subseteq\mathcal P(\N)$, we say that $C_1$ is {\em stronger than} $C_2$ (or $C_2$ is {\em weaker than} $C_1$), if every $C_1$-constrained map between chromatic metric pairs is necessarily $C_2$-constrained. We denote this fact by $C_1\succeq C_2$.
\end{definition}

Note that a map $f\colon\extMP{1}\to\extMP{2}$ between chromatic metric pairs is $\sigma$-constrained if and only if it is a pair map between $(A_1, \chi_1^{-1}(\sigma))$ and $(A_2, \chi_2^{-1}(\sigma))$. In particular, $f$ is $\{\N\}$-constrained if and only if $f(X_1)\subseteq X_2$.

\begin{remark}\label{remark:always-include-N}
    Even though by definition we do not require the maps between chromatic metric pairs to be pair maps, in the rest of the paper, we always require $\N$ to be part of any constraint set $C$, which has the same effect. We could alternatively
    also assume that $\bigcup C = \N$.
    This fact and our choice are further justified by the results in Section~\S\ref{sub:constraints_vs_topologies}. This is a matter of convenience rather than substance, we could equivalently lift any constraints on $C$ and insist that maps between chromatic metric pairs be pair maps.
\end{remark}

\begin{definition}
	Let $\N\in C\subseteq\mathcal P(\N)$ be a constraint set. A map $f\colon\MP{1}\to\MP{2}$ between chromatic metric pairs is said to be a {\em $C$-constrained isomorphism} if it is an isometry and both $f$ and $f^{-1}$ are $C$-constrained; equivalently, if $f\colon A_1\to A_2$ is an isometry such that, for every $\sigma\in C$, $f(\chi_1^{-1}(\sigma))=\chi_2^{-1}(\sigma)$. In this case, we say that $\MP{1}$ and $\MP{2}$ are {\em $C$-constrainedly isomorphic}. 
	
    Furthermore, a map $g\colon\MP{1}\to\MP{2}$ is a {\em $C$-constrained isomorphic embedding} if the corestriction to the image subspace,
    $g\colon (A_1, \chi_1)\to (g(A_1), \chi_2|_{X_2\cap g(A_1)})$, is a $C$-constrained isomorphism.
    
\end{definition}

Note that the composition of $C$-constrained chromatic metric pair maps is still $C$-constrained and the identity is always a $C$-constrained isomorphism for any $C\subseteq \mathcal{P}(\N)$.

\begin{remark}\label{rem:from_metric_to_chromatic_metric_pairs}
    A metric space $A$ can also be viewed as a metric pair $(A,A)$, and a metric pair $(A,X)$ can also be viewed as a chromatic metric pair $(A, 0_X)$, where $0_X\colon X\rightarrow\N$ is the constant zero map. 
    With these identifications, it is equivalent to say that two metric spaces are isometric, that the two corresponding metric pairs are pair isomorphic, and that the two corresponding chromatic metric pairs are $C$-constrained isomorphic for any choice of $C$; and analogously for two arbitrary metric pairs and their corresponding chromatic metric pairs.
\end{remark}

\begin{example}\label{ex:trivial_and_discrete_constraint_sets}
    Let $\extMP{1}$, $\extMP{2}$ be two chromatic metric pairs.
    \begin{compactenum}[(i)]
        \item As we always insist that $\N$ be in a constraint set, the minimal constraint set, which we call {\em trivial}, is $C_T=\{\N\}$. It is the weakest constraint set; it only forces the points of $X_1$ to be sent into $X_2$. 
        
        \item The set of all possible constraints, $C_D=\mathcal P(\N)$, is the strongest. We call it the {\em discrete} constraint set as it forces every color to be preserved. More precisely, a map $f\colon\MP{1}\to\MP{2}$ is $C_D$-constrained if and only if $\chi_2(f(x))=\chi_1(x)$ for every $x\in X_1$. Therefore,
        such map
        exists if and only if $\im\chi_1\subseteq\im\chi_2$,
        and, furthermore, 
        satisfies $\im\chi_1=\im(\chi_2|_{f(X_1)})$.
    \end{compactenum}
\end{example}

Trivially, if $C_1\supseteq C_2$, then $C_1$ is stronger than $C_2$. However, the converse implication does not hold in general. Consider, for example, the discrete constraint set and the set of all singletons, $C=\{\{n\}\mid n\in\N\}$. Both force every color to be fixed, so they have the same strength. 
Not so obviously, the same holds also for the cofinite constraint set $C^\prime=\{\sigma\subseteq\N\mid \N\setminus\sigma\text{ is finite}\}$. This relation is investigated in depth in the following subsection, \S\ref{sub:constraints_vs_topologies}. 

Before moving on, we conclude with an example where we discuss some of the notions introduced above.
\begin{example}\label{ex:examples_of_C-constrained_maps}
Let us consider four chromatic metric spaces, represented in Figure \ref{fig:example_cmp}. We discuss the existence of $C$-constrained maps between them for different choices of the constraint set $C$. Let $X$ be a disk in the plane with the Euclidean metric. Fix two disjoint annuli $Y$ and $Z$ of the same size contained in $X$. We consider four different coloring functions $\chi_1,\chi_2,\chi_3,\chi_4\colon X\to\N$ defined as follows: 
\[
\chi_1\colon X\mapsto 0,\quad\chi_2\colon\begin{cases}
    \begin{aligned}  Y\,&\mapsto 1\\
    X\setminus Y\,&\mapsto 0,
    \end{aligned}
\end{cases}\quad\chi_3\colon \begin{cases}\begin{aligned}  Y\,&\mapsto 1\\
Z\,&\mapsto 2\\
     X\setminus (Y\cup Z)\,&\mapsto 0,
    \end{aligned}\end{cases} \quad \text{and}\quad \chi_4\colon X\mapsto 2
\]
(in Figure \ref{fig:example_cmp}, the colors $0$, $1$ and $2$ are represented in gray, magenta and cyan, respectively).
    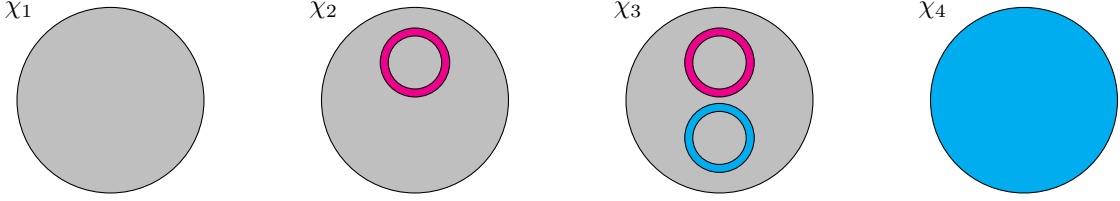
\begin{figure}[h!]
    \centering
    \begin{subfigure}{0.245\textwidth}
        \begin{tikzpicture}
        \centering
            \fill[lightgray] (0,0) circle (35pt);
            \draw (0,0) circle (35pt);
            \draw (-1.2,1.2) node{$\chi_1$};
        \end{tikzpicture}
    \end{subfigure}    
        \begin{subfigure}{0.245\textwidth}
        \begin{tikzpicture}
        \centering
            \fill[lightgray] (0,0) circle (35pt);
            \fill[magenta] (0,0.5) circle (13pt);
            \fill[lightgray] (0,0.5) circle (10pt);
            \draw (0,0) circle (35pt);
            \draw (0,0.5) circle (13pt);
            \draw (0,0.5) circle (10pt);
            \draw (-1.2,1.2) node{$\chi_2$};
        \end{tikzpicture}
    \end{subfigure}  
        \begin{subfigure}{0.245\textwidth}
        \begin{tikzpicture}
        \centering
            \fill[lightgray] (0,0) circle (35pt);
            \fill[magenta] (0,0.5) circle (13pt);
            \fill[lightgray] (0,0.5) circle (10pt);
            \fill[cyan] (0,-0.5) circle (13pt);
            \fill[lightgray] (0,-0.5) circle (10pt);
            \draw (0,0) circle (35pt);
            \draw (0,0.5) circle (13pt);
            \draw (0,0.5) circle (10pt);
            \draw (0,-0.5) circle (13pt);
            \draw (0,-0.5) circle (10pt);            \draw (-1.2,1.2) node{$\chi_3$};
        \end{tikzpicture}
    \end{subfigure}  
        \begin{subfigure}{0.245\textwidth}
        \begin{tikzpicture}
        \centering
            \fill[cyan] (0,0) circle (35pt);            \draw (0,0) circle (35pt);            
            \draw (-1.2,1.2) node{$\chi_4$};
        \end{tikzpicture}
    \end{subfigure} 
    \caption{\footnotesize A representation of the four chromatic metric spaces described in Example \ref{ex:examples_of_C-constrained_maps}.}
    \end{figure}\label{fig:example_cmp}

    Let us consider four different constraint sets: the trivial constraint set $C_T=\{\N\}$, the discrete constraint set  $C_D=\mathcal{P}(\N)$,
    $C_1=\{\{0\},\{0,1\},\{0,2\},\N\}$ and $C_2=\{\{1,2\},\N\}$. 
    
    With respect to the trivial constraint set $C_T$, all four chromatic metric spaces 
    are isomorphic (see Example \ref{ex:trivial_and_discrete_constraint_sets}(i)). 

    On the opposite side of the strength spectrum, for the discrete constraint set $C_D$, they are pairwise non-isomorphic since they use different sets of colors. Using the properties discussed in Example \ref{ex:trivial_and_discrete_constraint_sets}(ii), we can see that there are $C_D$-constrained maps only from $\chi_1$ to $\chi_2$ and $\chi_3$, from $\chi_2$ to $\chi_3$, and from $\chi_4$ to $\chi_3$. 

    More interesting are the impacts of $C_1$ and $C_2$.
    We discuss the case of $C_1$ in more detail.
    Since $C_D\succeq C_1$, the same maps mentioned above are $C_1$-constrained. The only difference is that a map from $\chi_4$ to $\chi_3$ can map $X$ into $X\setminus Y$, not just into $Z$ as in the case of $C_D$, relaxing the constraint. The identity maps from $\chi_2$, $\chi_3$ and $\chi_4$ to $\chi_1$, and from $\chi_3$ to $\chi_2$ are $C_1$-constrained maps. Furthermore, there are $C_1$-constrained maps from $\chi_4$ to $\chi_2$. They need to send $X$ into $X\setminus Y$.
    
    The following four diagrams encode the impact of the four different constraint sets. For each constraint set, we draw an arrow ($\to$) if there exists a constrained map, a hooked arrow ($\hookrightarrow$) if the 
    identity 
    is a constrained map, and an arrow pointing at both directions with a $\simeq$-symbol on top ($\xleftrightarrow{\simeq}$) if they are constrainedly isomorphic.
    
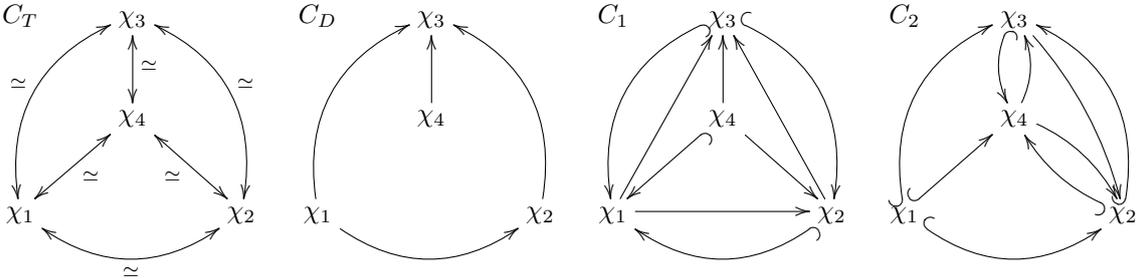
\begin{gather*}
	\xymatrix{
        C_T  & \chi_3\ar@{<->}^{\simeq}[d]\ar@{<->}@/^1.5pc/^{\simeq}[ddr]\ar@{<->}@/_1.5pc/_{\simeq}[ddl] & \\
         & \chi_4\ar@{<->}_{\simeq}[dr]\ar@{<->}^{\simeq}[dl] &\\
         \chi_1\ar@{<->}@/_1.5pc/_{\simeq}[rr] & & \chi_2
    }
    \quad
    \xymatrix{
        C_D  & \chi_3 & \\
         & \chi_4\ar[u] &\\
         \chi_1\ar@/_1.5pc/[rr]\ar@/^1.5pc/[uur] & & \chi_2\ar@/_1.5pc/[uul]
    }
    \quad
    \xymatrix{
        C_1  & \chi_3\ar@{^{(}->}@/_1.5pc/[ddl]\ar@{^{(}->}@/^1.5pc/[ddr] & \\
         & \chi_4\ar@{^{(}->}[dl]\ar[u]\ar[dr] &\\
         \chi_1\ar[rr]\ar[uur] & & \chi_2\ar@{^{(}->}@/^1.5pc/[ll]\ar[uul]
    }
    \quad
    \xymatrix{
        C_2  & \chi_3\ar@/^/[ddr]\ar@{^{(}->}@/_/[d] & \\
         & \chi_4\ar@/_/[u]\ar@/^/[dr] &\\
         \chi_1\ar@{^{(}->}@/_1.5pc/[rr]\ar@{^{(}->}[ur]\ar@{^{(}->}@/^1.5pc/[uur] & & \chi_2\ar@{^{(}->}@/_1.5pc/[uul]\ar@{^{(}->}@/^/[ul]
    }
    \end{gather*}
\end{example}

\subsection{The strength of constraint sets and Alexandrov topologies}\label{sub:constraints_vs_topologies}

Chromatic constraints can be viewed as topologies on the set of all colors, $\N$. In this section, we describe this relation in detail, and show how the strength of different sets of constraints can be compared.

Let $C\subseteq\mathcal P(\N)$. For every $n\in\N$, we define
\[
    C_n=\{\sigma\in C\mid n\in\sigma\},\quad\text{and}\quad \sigma_n=\bigcap C_n.
\]

Note that, if $C_n=\emptyset$, then $\sigma_n=\N$. Let us enlist a few immediate but crucial properties of these objects.

\begin{fact}\label{fact:C_n_and_sigma_n}
    Let $C\subseteq\mathcal P(\N)$. Then, the following properties hold:
    \begin{compactenum}[(i)]
        \item  For every $n\in\N$, $n\in\sigma_n$ and so the family $\Sigma_C=\{\sigma_n\}_{n\in\N}$ covers $\N$.
        \item For every $n,m\in\N$, if $m\in\sigma_n$, then $\sigma_m\subseteq\sigma_n$.
        \item For every $\tau\in C$, if $n\in\tau$, then $\sigma_n\subseteq\tau$.
        \item For every $\tau\in C$, $\tau=\bigcup_{n\in\tau}\sigma_n$.
    \end{compactenum}
\end{fact}
\begin{proof}
Item (i) is immediate, and item (ii) follows from the inclusion $C_m\supseteq C_n$. Since $\tau\in C_n$, item (iii) is trivial. Finally, items (i) and 
(iii) imply item (iv).
\end{proof}

Building on those basic facts, we discuss the topology generated by $\Sigma_C$, and observe that it is the Alexandrov topology generated by $C$.

\begin{proposition}\label{prop:T_C}
    Let $C\subseteq\mathcal P(\N)$. 
    \begin{compactenum}[(i)]
        \item The family $\Sigma_C=\{\sigma_n\}_{n\in\N}$ is a base of a topology $\Tau_C$ on $\N$.
        \item For every subset $U\subseteq\N$, $U\in\Tau_C$ if and only if it can be written as the particular union
        \[
            U=\bigcup_{n\in U}\sigma_n.
        \]
        \item The constraints are open in $\Tau_C$, i.e., $C\subseteq\Tau_C$.
        \item The topology $\Tau_C$ is an {\em Alexandrov topology}, i.e., arbitrary intersections of open subsets are open.
        \item $\Tau_C$ coincides with the smallest Alexandrov topology containing $C$.
    \end{compactenum}
\end{proposition}
\begin{proof}
    (i) According to Fact~\ref{fact:C_n_and_sigma_n}(i), it forms a cover of $\N$. Let $n,m\in\N$. If $\sigma_n\cap\sigma_m\neq\emptyset$, then for every $k\in\sigma_n\cap\sigma_m$, $k\in\sigma_k\subseteq\sigma_n\cap\sigma_m$ according to Fact~\ref{fact:C_n_and_sigma_n}(ii).

    (ii) If $U$ can be written as a union of elements of the base, it is trivially open. We want to show that if $U\in\Tau_C$, then $U\supseteq\bigcup_{n\in U}\sigma_n$ since the converse inclusion is clear.
    Let $n\in U$. Since $U$ is a union of elements of the base, there is $k\in U$ such that $n\in\sigma_k\subseteq U$. Hence, $\sigma_n\subseteq\sigma_k\subseteq U$ due to Fact~\ref{fact:C_n_and_sigma_n}(ii).

    (iii) It trivially follows from
    Fact~\ref{fact:C_n_and_sigma_n}(iv).

    (iv) Consider an arbitrary family $\mathcal{U}=\{U_i\}_{i\in I}$ of open subsets. Define $V=\bigcap\mathcal{U}$. For every $i\in I$, $U_i=\bigcup_{n\in U_i}\sigma_n$ according to item (ii). We claim that 
    \[
      V\supseteq\bigcup_{n\in V}\sigma_n,
    \]
    which concludes the proof. 
    Let $n\in V$. Then for every $i\in I$, there is $k_n^i\in U_i$ such that $n\in\sigma_{k_{n}^{i}}$. Hence, $\sigma_n\subseteq\sigma_{k_{n}^{i}}\subseteq U_i$, and so $\sigma_n\subseteq V$.

    (v) Clearly, for every $n\in\N$, $\sigma_n$ belongs to any Alexandrov topology containing $C$. Then, the result follows from items (ii), (iii) and (iv).
    
\end{proof}

The following proposition establishes the strength equivalence of three constraint sets.
 
\begin{proposition}\label{prop:C-constrained_iff_TC-constrained}
    Let $\N\in C\subseteq\mathcal P(\N)$ and  
    $f\colon\MP{1}\to\MP{2}$ be a map between chromatic metric pairs. Then, the following properties are equivalent:
    \begin{compactenum}[(i)]
        \item $f$ is $C$-constrained, 
        \item $f$ is $\{\sigma_n\}_{n\in\N}$-constrained, and
        \item $f$ is $\Tau_C$-constrained.
    \end{compactenum}
\end{proposition}
\begin{proof}
    The implication (iii)$\to$(i) is trivial.

    As for the implication (i)$\to$(ii), for every $n\in\N$,
    \[
        \begin{aligned}
        f(\chi_1^{-1}(\sigma_n))&\,=f(\chi_1^{-1}(\bigcap C_n)) =
        f(\bigcap_{\sigma\in C_n}\chi_1^{-1}(\sigma)) \subseteq
        \bigcap_{\sigma\in C_n}f(\chi_1^{-1}(\sigma)) \subseteq\\
        &\,\subseteq
        \bigcap_{\sigma\in C_n}\chi_2^{-1}(\sigma) =
        \chi_2^{-1}(\bigcap_{\sigma\in C_n}\sigma)=\chi_2^{-1}(\sigma_n).
        \end{aligned}
    \]
    
    Similarly, we obtain that $f$ is $\bigcup_{n\in I}\{\sigma_n\}$-constrained for every subset $I$ of $\N$, which shows the implication (ii)$\to$(iii). Indeed,
    \begin{equation}\label{eq:C-constrained_iff_TC-constrained}
        f(\chi_1^{-1}(\bigcup_{n\in I}\sigma_n)) =
        f(\bigcup_{n\in I}\chi_1^{-1}(\sigma_n)) =
        \bigcup_{n\in I}f(\chi_1^{-1}(\sigma_n)) \subseteq
        \bigcup_{n\in I}\chi_2^{-1}(\sigma_n) =
        \chi_2^{-1}(\bigcup_{n\in I}\sigma_n).
    \end{equation}
\end{proof}

Our goal is to characterize constraint sets up to their strengths. Proposition~\ref{prop:C-constrained_iff_TC-constrained} establishes that $C$ and $\Tau_C$ have the same strength. We would like to take advantage of this and compare the strength of constraint sets using inclusions of the Alexandrov topologies they generate. In order to do that, it remains to be checked that two constraint sets of the same strength always generate the same Alexandrov topology.
We consider the collection of all constraint sets with the same strength as $C$, and show that $\Tau_C$ is its maximum with respect to inclusion---in other words, if $\tau\subseteq\N$ is a constraint such that every $C$-constrained map is also $\tau$-constrained, then $\tau\in\Tau_C$.

\begin{theorem}\label{thm:all-constraints-are-alexandrov-topology}
    Let $\N\in C\subseteq\mathcal P(\N)$.
    Then $\Tau_C = \setdef{
            \sigma\subseteq\N
        }{
            \text{every $C$-constrained map is $\sigma$-constrained}
        }$.
\end{theorem}
\begin{proof}
    Proposition~\ref{prop:C-constrained_iff_TC-constrained} proves $(\subseteq)$. For $(\supseteq)$, let $\tau\in\mathcal P(\N)\setminus\Tau_C$. 
    We need to construct a map that is $C$-constrained, but not $\tau$-constrained. As $\tau\notin \Tau_C$, we have $\tau \subsetneq \bigcup_{n\in\tau}\sigma_n$ according to Proposition~\ref{prop:T_C}(ii). Then, there exists $n\in\tau$ and $m\in \sigma_n\setminus\tau$. Consider the identity map between two singleton chromatic metric spaces 
    \[
        f\colon(\{\ast\},\chi:\ast\mapsto n) \longrightarrow (\{\ast\},\psi:\ast\mapsto m).
    \]
    Then, $f$ is $C$-constrained, since $m\in\sigma_n$ and $\sigma_n$ is exactly the set of colors where $C$-constraint maps can send a point colored by $n$.
    But it is not $\tau$-constrained, since $n\in\tau$ and $m\notin\tau$.
\end{proof}

A characterization of the strength of constraint sets immediately follows.
\begin{corollary}\label{coro:constraints-strength-and-alexandrov-topology}
    Given two subsets $\N\in C,D\subseteq\mathcal P(\N)$, the constraint set $C$ is stronger than $D$, $C\succeq D$, if and only if $\Tau_C \supseteq \Tau_D$.
    They have the same strength (i.e., $C\preceq D$ and $C\succeq D$) if and only if $\Tau_C = \Tau_D$.
\end{corollary}

\section{Gromov-Hausdorff distance for chromatic metric pairs}\label{sec:C-constrained_GH}

Having set up the notion of $C$-constrained maps, we can define a variant of the Gromov-Hausdorff distance for chromatic metric pairs analogous to the definition via pairs of maps in Theorem~\ref{thm:Kalton_Ostrovskii}.
We give the definition for any constraint set, whether it contains $\N$ or not---recall Remark~\ref{remark:always-include-N} for more details.

\begin{definition}\label{def:C-constrained_GH}
	Let $\MP{1}$ and $\MP{2}$ be two chromatic metric pairs and $C\subseteq\mathcal P(\N)$. We define their {\em $C$-constrained Gromov-Hausdorff distance} as
	\[
	\GH^C(\MP{1},\MP{2})=\frac{1}{2}\inf_{\substack{f\colon A_1\to A_2\\g\colon A_2\to A_1\\\text{$C$-constrained}}}\max\{\dis f,\dis g,\codis(f,g)\}.
	\]
\end{definition}

For every choice of $C$, we get a different notion of distance. Clearly, two constraint sets of the same strength will yield the same distances. By the same token, a weaker constraint set yields at most as large distances as a stronger one, since there is more freedom for the choice of $f$ and $g$ to minimize over. One motivation for choosing a weaker $C$, with weaker distinguishing power, is the desire to bound distances between various invariants as tightly as possible---if an invariant does not use all the color information, we should choose a weaker constraint set $C$ and consequently get a smaller distance between the pairs.
Note, in particular, that the choice $C=\{\N\}$ yields a Gromov-Hausdorff distance between the pairs $(A_1, X_1), (A_2, X_2)$---see Corollary~\ref{coro:GH_metric_pairs} for comparison with the notion defined in \cite{GomChe}---and the choice $C=\emptyset$ yields the usual Gromov-Hausdorff distance between the spaces $A_1$ and $A_2$.

\begin{example}
    Let us consider the spaces studied in Example~\ref{ex:examples_of_C-constrained_maps}. The $C_T$-constrained Gromov-Hausdorff distance between the four spaces is trivially $0$, and the $C_D$-constrained one is infinite since, for no pair of spaces, there are $C_D$-constrained maps going in both directions. Consider $C_1$. Using the knowledge gathered in the example, we can easily conclude that
    \[
    \GH^{C_1}(\chi_i,\chi_j)<\infty,\quad \text{and}\quad \GH^{C_1}(\chi_i,\chi_4)=\infty,\quad\text{for $i,j=1,2,3$}.
    \]
    Similarly, we can show that
    \[
    \GH^{C_2}(\chi_i,\chi_j)<\infty,\quad \text{and}\quad \GH^{C_2}(\chi_1,\chi_i)=\infty,\quad\text{for $i,j=2,3,4$}.
    \]
\end{example}

Let us state a few immediate observations.
\begin{remark}\label{rem:C-constrained_GH_and_other_GH}
	\begin{compactenum}[(i)]
		\item The difference between Definition \ref{def:C-constrained_GH} and the characterization of the classical Gromov-Hausdorff distance provided in Theorem \ref{thm:Kalton_Ostrovskii} is that, in the former, we restrict the family of admissible functions. In particular, it is immediate to see that, in the notation of Definition \ref{def:C-constrained_GH},
		\[
    		\GH^C(\MP{1},\MP{2})\geq\max\{\GH(A_1,A_2),\GH(\chi_1^{-1}(\sigma),\chi_2^{-1}(\sigma))\mid\sigma\in C\}.
		\]
        It is not hard to construct examples where the inequality is strict. Consider $X_1=X_2=[-M,M]$ for some $M>0$, $m<M/2$, and define the following $\chi_1, \chi_2$ (see Figure \ref{fig:GH_not_enough}):
        \[
            \chi_1\colon\begin{cases}
                \begin{aligned}
                    [0,2m]&\,\mapsto 0\\
                    [-M,M]\setminus[0,2m]&\,\mapsto 1,
                \end{aligned}
            \end{cases}\quad\text{and}\quad
             \chi_2\colon\begin{cases}
                \begin{aligned}
                    [-m,m]&\,\mapsto 0\\
                    [-M,M]\setminus[-m,m]&\,\mapsto 1
                \end{aligned}
            \end{cases}
        \]
        \begin{figure}[h!]
            \centering
            \begin{tikzpicture}
                \draw[magenta,ultra thick] (0,0)--(2,0);
                \draw[cyan,ultra thick] (-2.5,0)--(0,0) (2,0)--(2.5,0);
                \draw 
                (0,0.1) node[above]{$0$};
                \fill[magenta] (0,0) circle (2pt);
                \draw 
                (2,0.1) node[above]{$2m$};
                \fill[magenta] (2,0) circle (2pt);
                \draw 
                (-2.5,0.1) node[above]{$-M$}; 
                \fill[cyan] (-2.5,0) circle (2pt);
                \draw 
                (2.5,0.1) node[above]{$M$}; 
                \fill[cyan] (2.5,0) circle (2pt);
                \draw (0,-0.3)node[below]{$\chi_1$};

                \draw[magenta,ultra thick] (6,0)--(8,0);
                \draw[cyan,ultra thick] (4.5,0)--(6,0) (8,0)--(9.5,0);
                \draw 
                (6,0.1) node[above]{$-m$};
                \fill[magenta] (6,0) circle (2pt);
                \draw 
                (8,0.1) node[above]{$m$};
                \fill[magenta] (8,0) circle (2pt);
                \draw 
                (4.5,0.1) node[above]{$-M$}; 
                \fill[cyan] (4.5,0) circle (2pt);
                \draw 
                (9.5,0.1) node[above]{$M$};
                \fill[cyan] (9.5,0) circle (2pt);
                \draw (7,-0.3)node[below]{$\chi_2$};
            \end{tikzpicture}
            \caption{\footnotesize A representation of the two chromatic metric spaces defined in Remark \ref{rem:C-constrained_GH_and_other_GH}(i).}\label{fig:GH_not_enough}
        \end{figure}
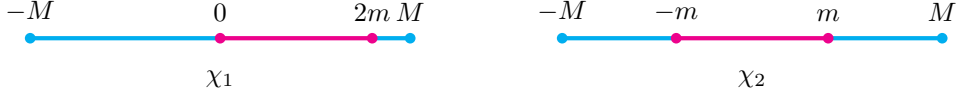
        
        Set $C=\{\{0\},\N\}$. Then, $\GH^{C}(\chi_1,\chi_2)>0$, but $\GH(X_1,X_2)=\GH(\chi_1^{-1}(\{0\}),\chi_2^{-1}(\{0\}))=0$ since both pairs are isometric.
		
        \item We have discussed in Remark~\ref{rem:from_metric_to_chromatic_metric_pairs} that metric spaces and metric pairs can be canonically seen as chromatic metric pairs. Those constructions agree with the $C$-constrained Gromov-Hausdorff distance. Indeed, for every pair of metric spaces $X$ and $Y$ and every
        constraint set $C\subseteq\mathcal P(\N)$,
	\[
		\GH(X,Y)=\GH^C((X,
        0_X)
        ,(Y,
        0_Y)
        ).
		\]
		Furthermore, if $\chi\colon X\to\N$ and $\psi\colon Y\to\N$ are two arbitrary coloring functions, then
		\[
		\GH(X,Y)=\GH^{\{\N\}}((X,\chi \colon X \to \N),(Y,\psi \colon Y \to \N)).
        \]
        Using the same idea, we can induce a notion of Gromov-Hausdorff distance on metric pairs: if $(A_1,X_1)$ and $(A_2,X_2)$ are metric pairs, then define their {\em Gromov-Hausdorff distance} as
        \[
        \GH((A_1,X_1),(A_2,X_2))=\GH^{\{\N\}}((A_1,0_{X_1}),(A_2,0_{X_2})).
        \]
        This distance notion will be compared to that introduced in \cite{GomChe} in Corollary~\ref{coro:GH_metric_pairs}.
		
        \item Since compositions of $C$-constrained maps are $C$-constrained, 
        the triangle inequality of the $C$-constrained Gromov-Hausdorff distance can be shown as in the classic case. It is obvious that it is also symmetric.

        \item Let $\N\in C_i\subseteq\mathcal P(\N)$ for $i=1,2$ be two constraint sets. If $C_1\preceq C_2$, that is, every $C_2$-constrained map is also $C_1$-constrained, then, for every pair of chromatic metric pairs $\MP{1}$ and $\MP{2}$,
        \[
        \GH^{C_1}(\MP{1},\MP{2})\leq\GH^{C_2}(\MP{1},\MP{2}).
        \]
		
        \item 
        In the notation of Definition \ref{def:C-constrained_GH}, 
		\[
    		\GH^C(\MP{1},\MP{2})
            =\GH^{\Sigma_C}(\MP{1},\MP{2})
            =\GH^{\Tau_C}(\MP{1},\MP{2}),
		\]
		where $\Sigma_C=\{\sigma_n\}_{n\in\N}$. This holds because Proposition \ref{prop:C-constrained_iff_TC-constrained} implies that the three constraint sets allow for the same maps. In fact, the distance $\GH^{C'}$ is the same for any $C'$ such that $C\subseteq C' \subseteq \Tau_C$ or $\Sigma_C\subseteq C' \subseteq \Tau_C$.
	\end{compactenum}
\end{remark}

\subsection{Invariants and computations}\label{sub:invariants}

In \cite{Mem12}, the author collected some {\em metric invariants} assigning to every metric space $X$ a value $\psi(X)$ in some other metric space such that, if $X$ and $Y$ are isometric metric spaces, then $\psi(X)=\psi(Y)$. Furthermore, an important property of these invariants is that they are {\em stable}, i.e., the distance between two invariants $\psi(X)$ and $\psi(Y)$ is bounded by some function of the Gromov-Hausdorff distance between $X$ and $Y$.
The diameter of a metric space is an example that we already recalled in Proposition~\ref{prop:GH_and_diam}.

To provide lower bounds to the $C$-constrained Gromov-Hausdorff distance, we can use Remark \ref{rem:C-constrained_GH_and_other_GH}(i) in conjunction with stable metric invariants. Suppose that $\psi$ is a metric invariant with values in $(\mathcal Y,d_{\mathcal Y})$ satisfying, for some $k>0$,
\[
d_{\mathcal Y}(\psi(X),\psi(Y))\leq k\cdot\GH(X,Y)
\]
for every pair of metric spaces $X$ and $Y$. 
Then, for every $\N\in C\subseteq\mathcal P(\N)$ and every pair of chromatic metric pairs $\MP{1}$ and $\MP{2}$,
\begin{equation*} \label{eq:lower-bound_1}\GH^C(\MP{1},\MP{2})\geq\GH(A_1,A_2)\geq\frac{1}{k}d_{\mathcal Y}(\psi(A_1),\psi(A_2))
\end{equation*}
and, for every $\sigma\in C$,
\begin{equation*} \label{eq:lower-bound_2}
    \GH^C(\MP{1},\MP{2})\geq\GH(\chi_1^{-1}(\sigma),\chi_2^{-1}(\sigma))\geq\frac{1}{k}d_{\mathcal Y}(\psi(\chi_1^{-1}(\sigma)),\psi(\chi_2^{-1}(\sigma))).
\end{equation*}

In the following example, we make use of these ideas to provide exact computations of the $C$-constrained Gromov-Hausdorff distance.

\begin{example}\label{ex:C-GH_computations}
    
    Let us consider four chromatic metric spaces: $([0,3r),\chi_1)$, $([0,4r),\chi_2)$, $([0,5r),\chi_3)$ and $([0,3r),\chi_4)$ where the intervals are endowed with the usual Euclidean metric and the coloring maps are
    \[
        \chi_1\colon[0,3r)\mapsto 0,\quad\chi_2\colon\begin{cases}
            \begin{aligned}
                [0,r) &\mapsto 1\\
                [r,4r) &\mapsto 0,
            \end{aligned}
        \end{cases}
        \quad \chi_3\colon\begin{cases}
            \begin{aligned}
                [0,r)\cup[4r,5r) &\mapsto 1\\
                [r,4r) &\mapsto 0,
            \end{aligned}
        \end{cases}
        \quad\text{and}\quad\chi_4\colon[0,3r)\mapsto 1.
    \]
    (see Figure \ref{fig:C-GH_computations}). We want to compute the $C$-constrained Gromov-Hausdorff distances between them for $C_D$ and $C=\{\{0\},\N\}$.
    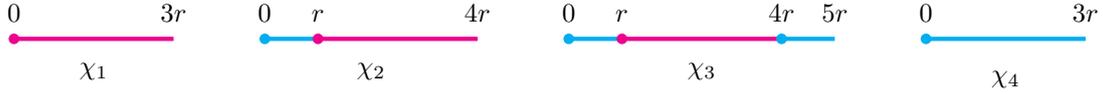
\begin{figure}[h!]
        \centering
        \begin{tikzpicture}
            \draw[magenta, ultra thick] (0,0)--(2.1,0);
            \draw (1.05,-0.2) node[below]{$\chi_1$};
            \draw[magenta, ultra thick] (4,0)--(6.1,0);
            \draw[cyan, ultra thick] (3.3,0)--(4,0);
            \draw (4.7,-0.2) node[below]{$\chi_2$};
            \draw[magenta, ultra thick] (8,0)--(10.1,0);
            \draw[cyan, ultra thick] (7.3,0)--(8,0) (10.1,0)--(10.8,0);
            \draw (9.05,-0.2) node[below]{$\chi_3$};
            \draw[cyan, ultra thick] (12,0)--(14.1,0);
            \draw (13.05,-0.3) node[below]{$\chi_4$};
            \draw 
            (0,0.1) node[above]{$0$};
            \fill[magenta] (0,0) circle (2pt);
            \draw 
            (2.1,0.1) node[above]{$3r$};
            \draw 
            (3.3,0.1) node[above]{$0$};
            \fill[cyan] (3.3,0) circle (2pt);
            \draw 
            (4,0.1) node[above]{$r$};
            \fill[magenta] (4,0) circle (2pt);
            \draw 
            (6.1,0.1) node[above]{$4r$};
            \draw 
            (7.3,0.1) node[above]{$0$};
            \fill[cyan] (7.3,0) circle (2pt);
            \draw 
            (8,0.1) node[above]{$r$};
            \fill[magenta] (8,0) circle (2pt);
            \draw 
            (10.1,0.1) node[above]{$4r$};
            \fill[cyan] (10.1,0) circle (2pt);
            \draw 
            (10.8,0.1) node[above]{$5r$};
            \draw 
            (12,0.1) node[above]{$0$};
            \fill[cyan] (12,0) circle (2pt);
            \draw 
            (14.1,0.1) node[above]{$3r$};
        \end{tikzpicture}
        \caption{\footnotesize A representation of the four chromatic metric pairs defined in Example \ref{ex:C-GH_computations}(ii). The points colored in $0$ are represented in magenta, and in cyan, those colored in $1$.}
        \label{fig:C-GH_computations}
    \end{figure}

    First, we can immediately note that
    \begin{gather*}
    \GH^{C_D}(\chi_1,\chi_2)=\GH^{C_D}(\chi_1,\chi_3)=\GH^{C_D}(\chi_1,\chi_4)=\GH^{C_D}(\chi_2,\chi_4)=\GH^{C_D}(\chi_3,\chi_4)=\infty,\quad\text{and}\\
    \GH^C(\chi_1,\chi_4)=\GH^C(\chi_2,\chi_4)=\GH^C(\chi_3,\chi_4)=\infty.
    \end{gather*}
    To lower bound the other distances, we use Remark \ref{rem:C-constrained_GH_and_other_GH}(i) and Proposition \ref{prop:GH_and_diam}. Using the two results, we obtain the following inequalities:
    \begin{gather*}
    \GH^{C_D}(\chi_2,\chi_3)\geq\frac{\lvert\diam\chi_2^{-1}(\{1\})-\diam\chi_3^{-1}(\{1\})\rvert}{2}=\frac{\lvert r-5r\rvert}{2}=2r,\\
    \GH^C(\chi_1,\chi_2)\geq \frac{\lvert\diam\chi_1^{-1}(\N)-\diam\chi_2^{-1}(\N)\rvert}{2}=\frac{r}{2},\quad \GH^C(\chi_1,\chi_3)\geq r,\quad\text{and}\quad \GH^C(\chi_2,\chi_3)\geq \frac{r}{2}.
    \end{gather*}
    We construct particular constrained maps to provide upper bounds. 
    
    Let us start with $\GH^{C_D}(\chi_2,\chi_3)$. We consider the inclusion map $i\colon \chi_2\to\chi_3$, whose distortion is $0$. In the converse direction, both segments $[0,r)$ and $[4r,5r)$ must be sent into $[0,r)$. Define a $C_D$-constrained map $f\colon\chi_3\to\chi_2$ 
    as follows: for every $x\in[0,5r)$,
    \[
    f(x)=\begin{cases}
        \begin{aligned}
            & x & \text{if $x\in[0,4r)$,}\\
            & x-4r & \text{if $x\in [4r,5r)$.}
        \end{aligned}
    \end{cases}
    \]
    Then, $\dis f=\codis(f,i)=4r$, and $\GH^{C_D}(\chi_2,\chi_3)\leq 2r$. Therefore, $\GH^{C_D}(\chi_2,\chi_3)=2r$. 

    Let us estimate the $C$-constrained Gromov-Hausdorff distance between $\chi_1$, $\chi_2$ and $\chi_3$. In all three cases, in one direction, we can consider the inclusion maps. In the opposite direction, we define maps $f^j_i\colon\chi_j\to\chi_i$ for $1\leq i<j\leq 3$ as follows:
    \[
    f^2_1(x)=\begin{cases}
        \begin{aligned}
            & 0 & \text{if $x\in[0,r)$,}\\
            & x-r & \text{if $x\in[r,4r)$,}
        \end{aligned}
    \end{cases}\quad
    f^3_1(x)=\begin{cases}
        \begin{aligned}
            & 0 & \text{if $x\in[0,r)$,}\\
            & x-r & \text{if $x\in[r,4r)$,}\\
            & 3r & \text{if $x\in[4r,5r)$,}
        \end{aligned}
    \end{cases}\quad\text{and}\quad
    f^3_2(x)=\begin{cases}
        \begin{aligned}
            & x & \text{if $x\in[0,4r)$,}\\
            & 4r & \text{if $x\in[4r,5r)$.}\\
        \end{aligned}
    \end{cases}
    \]
    
    By computing the distortion and codistortion of these maps, we can conclude that
    \[
    \GH^C(\chi_1,\chi_2)=\frac{r}{2},\quad\GH^C(\chi_1,\chi_3)=r,\quad\text{and}\quad \GH^C(\chi_2,\chi_3)=\frac{r}{2}.
    \]
    We want to emphasize that
    \[
    \GH^C(\chi_2,\chi_3)=\frac{r}{2}<2r=\GH^{C_D}(\chi_2,\chi_3),
    \]
    which is consistent with Remark \ref{rem:C-constrained_GH_and_other_GH}(iv).
    
\end{example}

Inspired by known metric invariants, we can define similar concepts in the realm of chromatic metric pairs. The goal is to obtain more discerning tools providing tighter lower bounds.

\begin{definition}\label{def:cchromatic_invariants}
    Let $\N\in C\subseteq\mathcal P(\N)$.
    For every chromatic metric pair $\MP{}$, and every pair of distinct $\sigma,\tau\in\mathcal U$, we define the function
    \[
        \mathcal L_{\MP{}}^{\sigma,\tau}\colon\chi^{-1}(\sigma)\to\mathcal P(\R_{\geq 0}),
        \text{ by letting }\mathcal L_{\MP{}}^{\sigma,\tau}(x)=\{d(x,x')\mid x'\in\chi^{-1}(\tau)\}
        \text{, for every $x\in\chi^{-1}(\sigma)$,}
    \]
    be the collection of all distances from $x$ to $\tau$-colored points.
    This is a generalization of the local distance set as presented in \cite{Mem12}.
    We also define more invariants, inspired by the notions of the distance set, the eccentricity function, the separation, and the circumradius, which are all classical invariants of the Gromov-Hausdroff distance, collected in \cite{Mem12}, and give more intuitive descriptions below:
    \begin{compactenum}[-]
        \item $\mathcal D_{\MP{}}^{\sigma,\tau}=\bigcup\mathcal L_{\MP{}}^{\sigma,\tau}(\chi^{-1}(\sigma))$;
        
        \item $\ecc_{\MP{}}^{\sigma,\tau}\colon\chi^{-1}(\sigma)\to\R_{\geq 0}$, 
        given by 
        $\ecc_{\MP{}}^{\sigma,\tau}(x)=\sup\mathcal L_{\MP{}}^{\sigma,\tau}(x) = \sup_{x'\in\chi^{-1}(\tau)}d(x,x')$, for $x\in\chi^{-1}(\sigma)$;
        
        \item $\sep_{\MP{}}^{\sigma,\tau}\colon\chi^{-1}(\sigma)\to\R_{\geq 0}$, given by 
        $\sep_{\MP{}}^{\sigma,\tau}(x)=\inf\mathcal L_{\MP{}}^{\sigma,\tau}(x)
        = \inf_{x'\in\chi^{-1}(\tau)}d(x,x')$, for $x\in\chi^{-1}(\sigma)$;
        
        \item $\mathcal E_{\MP{}}^{\sigma,\tau}=\ecc_{\MP{}}^{\sigma,\tau}(\chi^{-1}(\sigma))$;
        
        \item $\mathcal S_{\MP{}}^{\sigma,\tau}=\sep_{\MP{}}^{\sigma,\tau}(\chi^{-1}(\sigma))$;
        
        \item $\rad_{\MP{}}^{\sigma,\tau}=\inf\mathcal E_{\MP{}}^{\sigma,\tau}
         = \inf_{x\in\chi^{-1}(\sigma)}\sup_{x'\in\chi^{-1}(\tau)}d(x,x')$;
        
        \item $\dist_{\MP{}}^{\sigma,\tau}=\inf\mathcal S_{\MP{}}^{\sigma,\tau}
         = \inf_{x\in\chi^{-1}(\sigma)}\inf_{x'\in\chi^{-1}(\tau)}d(x,x')$.
    \end{compactenum}
    
\end{definition}

The intuitive meaning of the invariants defined above is as follows.
For every $\sigma$-colored point $x\in X$:
\begin{compactenum}[-]
    \item $\mathcal L_{\MP{}}^{\sigma,\tau}(x)$ collects all its distances to points that are $\tau$-colored;
    \item $\mathcal D^{\sigma,\tau}_{\MP{}}$ consists of all the lengths of edges connecting a $\sigma$-colored point and a $\tau$-colored point;
    \item $\ecc_{\MP{}}^{\sigma,\tau}(x)$ is the radius of the smallest closed 
    ball centered in $x$ containing all the $\tau$-colored points;
    \item $\sep_{\MP{}}^{\sigma,\tau}(x)$ is the distance between $x$ and the set $\chi^{-1}(\tau)$,
    \item $\rad_{\MP{}}^{\sigma,\tau}$ is the radius of the smallest closed ball centered in a $\sigma$-colored point and containing all the $\tau$-colored points;
    \item $\dist_{\MP{}}^{\sigma,\tau}$ is
    
    the distance $\dist(\chi^{-1}(\sigma),\chi^{-1}(\tau))$, with the usual notation $\dist(A,B)=\inf\{d(a,b)\mid a\in A,b\in B\}$ for two subsets $A$ and $B$ of a common metric space.
\end{compactenum}

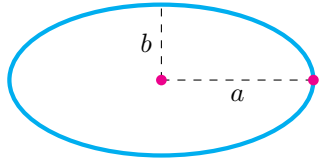
\begin{figure}[!h]
    \centering
    \begin{tikzpicture}
                \draw[dashed] (0,0)--(2,0)node[pos=0.5,below]{$a$};
                \draw[dashed] (0,0)--(0,1)node[pos=0.5,left]{$b$};
                \draw[cyan, ultra thick] (0,0) ellipse (2cm and 1cm);
                \fill[magenta] (0,0) circle (2pt) (2,0) circle (2pt);
                
            \end{tikzpicture}
    \caption{\footnotesize The chromatic metric pair $(\mathbb R^2, \chi)$ as defined in Example \ref{ex:invariants}. The colors $0$ and $1$ are represented in magenta and cyan, respectively.}
    \label{fig:invariants}
\end{figure}

\begin{example} \label{ex:invariants}
    Let us discuss 
    a quick example to demonstrate the different invariants defined above. Take the chromatic metric pair $(\mathbb R^2, \chi\colon X \to \mathbb \{0,1\})$ endowed with the usual Euclidean metric and where $X$ is given by the disjoint union of an ellipse with semi-major axis $a$ and semi-minor axis $b\leq a$ together with its center point at $(0,0)$. The coloring is
    \[ \chi(x) = \begin{cases}
        0 & \text{if }x\in\{(0,0), (a,0)\}, \\
        1 & \text{otherwise},
    \end{cases} 
    \]
    as shown 
    in Figure \ref{fig:invariants}.  
    We have: 
    \begin{align*}
        \mathcal L_{(\mathbb R^2, \chi)}^{\{0\},\{1\}}((0,0)) &= [b,a], &
        \mathcal L_{(\mathbb R^2, \chi)}^{\{0\},\{1\}}((a,0)) &= (0,2a], &
        \mathcal D_{(\mathbb R^2, \chi)}^{\{0\},\{1\}} &= (0,2a], &
        &
        \\
        \ecc_{(\mathbb R^2, \chi)}^{\{0\},\{1\}}((0,0)) &= a, &
        \ecc_{(\mathbb R^2, \chi)}^{\{0\},\{1\}}((a,0)) &= 2a, &
        \mathcal E_{(\mathbb R^2, \chi)}^{\{0\},\{1\}} &= \{a,2a\}, &
        \rad_{(\mathbb R^2, \chi)}^{\{0\},\{1\}} &= a,
        \\
        \sep_{(\mathbb R^2, \chi)}^{\{0\},\{1\}}((0,0)) &= b, &
        \sep_{(\mathbb R^2, \chi)}^{\{0\},\{1\}}((a,0)) &= 0, &
        \mathcal S_{(\mathbb R^2, \chi)}^{\{0\},\{1\}} &= \{0,b\}, &
        \dist_{(\mathbb R^2, \chi)}^{\{0\},\{1\}} &= 0,
    \end{align*}
    with the first column corresponding to the point $(0, 0)$, the second to the point $(a, 0)$, the third column is the union of the previous two, and the last is the infimum of the third.
\end{example}

All the objects defined in Definition \ref{def:cchromatic_invariants} are stable with respect to the $C$-constrained Gromov-Hausdorff distance as we formally state in the following result.

\begin{proposition}\label{prop:chromatic_invariants}
    Let $\N\in C\subseteq\mathcal P(\N)$. For every pair of chromatic metric pairs $\MP{1}$ and $\MP{2}$, and every pair of distinct $\sigma,\tau\in\mathcal \Tau_C$, we consider the following invariant
    \[
        d_{\mathcal{L}}^{\sigma,\tau}((A_1,\chi_1), (A_2, \chi_2))
        =
        \inf_{R\in\mathcal R(\chi_1^{-1}(\sigma),\chi_2^{-1}(\sigma))}\sup_{(x_1,x_2)\in R}d_H(\mathcal L_{\MP{1}}^{\sigma,\tau}(x_1),\mathcal L_{\MP{2}}^{\sigma,\tau}(x_2)).
    \]
    Then, the following inequalities hold:
    \begin{equation}
    \begin{aligned}
        \lvert\rad_{\MP{1}}^{\sigma,\tau}-\rad_{\MP{2}}^{\sigma,\tau}\rvert&\,
        \leq
        d_H(\mathcal E_{\MP{1}}^{\sigma,\tau},\mathcal E_{\MP{2}}^{\sigma,\tau})
        = \\
        &\,=
        \inf_{R\in\mathcal R(\chi_1^{-1}(\sigma),\chi_2^{-1}(\sigma))}\sup_{(x_1,x_2)\in R}\lvert\ecc_{\MP{1}}^{\sigma,\tau}(x_1)-\ecc_{\MP{2}}^{\sigma,\tau}(x_2)\rvert\leq \\
        &\,\leq
        d_{\mathcal{L}}^{\sigma,\tau}((A_1,\chi_1), (A_2, \chi_2)),
    \end{aligned}
    \end{equation}
    \begin{equation}
    \begin{aligned}
        \lvert\dist_{\MP{1}}^{\sigma,\tau}-\dist_{\MP{2}}^{\sigma,\tau}\rvert&\,
        \leq
        d_H(\mathcal S_{\MP{1}}^{\sigma,\tau},\mathcal S_{\MP{2}}^{\sigma,\tau})
        =\\
        &\,=
        \inf_{R\in\mathcal R(\chi_1^{-1}(\sigma),\chi_2^{-1}(\sigma))}\sup_{(x_1,x_2)\in R}\lvert\sep_{\MP{1}}^{\sigma,\tau}(x_1)-\sep_{\MP{2}}^{\sigma,\tau}(x_2)\rvert
        \leq \\
        &\,\leq
        d_{\mathcal{L}}^{\sigma,\tau}((A_1,\chi_1), (A_2, \chi_2)),
    \end{aligned}
    \end{equation}
    
    \begin{equation}\label{eq:stable_chromatic_invariants}
    \begin{aligned}
    d_H(\mathcal D_{\MP{1}}^{\sigma,\tau},\mathcal D_{\MP{2}}^{\sigma,\tau})
    &\,\leq
    d_{\mathcal{L}}^{\sigma,\tau}((A_1,\chi_1), (A_2, \chi_2))
    \leq \\
    &\,\leq
    2\GH^C(\chi_1,\chi_2)\leq 2\GH^C(\MP{1},\MP{2}),
    \end{aligned}
    \end{equation}
    where, we recall, $\chi_i$ denotes the chromatic metric space $(\dom\chi_i, \chi_i)$.
\end{proposition}
\begin{proof}
Most of the inequalities are easy to verify---relying on the observation that the difference between the infima (or suprema) of two subsets of $\R$ are always upper-bounded by their Hausdorff distance. For the equalities, a correspondence on either side can be used to construct a correspondence on the other such that the differences considered under the suprema are the same.

Let us focus on the most relevant inequality for our purposes, i.e., the second one in \eqref{eq:stable_chromatic_invariants}. Fix $\varepsilon>0$, and two $C$-constrained maps $f\colon\chi_1\to\chi_2$ and $g\colon\chi_2\to\chi_1$ 
such that
\[
    \max\{\dis f,\dis g,\codis(f,g)\}\leq 2\GH^C(\chi_1,\chi_2)+\varepsilon.
\]
Define $R=\gr(f|_{\chi_1^{-1}(\sigma)})\cup(\gr(g|_{\chi_2^{-1}(\sigma)}))^{-1}$, which is a correspondence between $\chi_1^{-1}(\sigma)$ and $\chi_2^{-1}(\sigma)$.
Fix $(x_1,x_2)\in R$. By definition, either $f(x_1)=x_2$ or $x_1=g(x_2)$. Let us consider the former case; the latter can be treated analogously. We upper-bound the Hausdorff distance between $\mathcal L_{\MP{1}}^{\sigma,\tau}(x_1)$ and $\mathcal L_{\MP{2}}^{\sigma,\tau}(x_2)$ by considering the correspondence \[
    S = \setdef{(d_1(x_1, x'), d_2(x_2, f(x')))}{x'\in\chi_1^{-1}(\tau)} \cup \setdef{(d_1(x_1, g(y')), d_2(x_2, y'))}{y'\in\chi_2^{-1}(\tau)}.
\] 
For every $x'\in\chi_1^{-1}(\tau)$ and $y'\in\chi_2^{-1}(\tau)$, we have
\begin{align*}
&\lvert d_1(x_1,x')-d_2(x_2,f(x'))\rvert = \lvert d_1(x_1,x')-d_2(f(x_1),f(x'))\rvert \leq \dis f,
\\
&\lvert d_1(x_1,g(y'))-d_2(x_2,y')\rvert = \lvert d_1(x_1,g(y'))-d_2(f(x_1),y')\rvert \leq \codis(f,g),
\end{align*}
and so 
\begin{equation}\label{eq:stability}
    d_H(\mathcal L_{\MP{1}}^{\sigma,\tau}(x_1),\mathcal L_{\MP{2}}^{\sigma,\tau}(x_2))\leq\max\{\dis f,\codis(f,g)\}.
\end{equation}
For the analogous inequality in the case where $x_1=g(x_2)$, we need to use $\dis g$ in the argument of the maximum on the right-hand side of \eqref{eq:stability}.

Hence, we have shown that
\[
    \sup_{(x_1,x_2)\in R}d_H(\mathcal L_{\MP{1}}^{\sigma,\tau}(x_1),\mathcal L_{\MP{2}}^{\sigma,\tau}(x_2))\leq\max\{\dis f,\dis g,\codis(f,g)\}\leq 2\GH^C(\chi_1,\chi_2)+\varepsilon.
\]
Since $\varepsilon$ can be taken arbitrarily small, we have obtained the desired inequality.
\end{proof}

\begin{example}\label{ex:chromatic_invariants}
    Let us show examples where these new invariants can be applied to compute the $C$-constrained Gromov-Hausdorff distance. We intend to provide instances where even the weakest invariants (i.e., those with weaker discerning power according to Proposition \ref{prop:chromatic_invariants}) provide optimal lower bounds.

\begin{compactenum}[(i)]    
\item We construct two chromatic metric spaces $\chi_1$ and $\chi_2$ and a constraint set $\N\in C\subseteq\mathcal P(\N)$ such that
\begin{compactenum}[(a)]
\item $\chi_1^{-1}(\sigma)$ and $\chi_2^{-1}(\sigma)$ are isometric for every $\sigma\in C$, and
\item there are two elements $\sigma$ and $\tau$ of an open partition of $(\N,\Tau_C)$ such that $\GH^C(\chi_1,\chi_2)=\lvert\rad^{\sigma,\tau}_{\chi_1}-\rad^{\sigma,\tau}_{\chi_2}\rvert/2$.
\end{compactenum}

Consider the following two chromatic metric spaces: take the subspace $X$ of $\R^2$ endowed with the $1$-norm defined by
    \[X'=[(0,0),(1,1)]\cup[(1,1),(0,2)], \quad X''= [(1,0),(2,1)]\cup[(2,1),(1,2)],\quad \text{and}\quad X=X'\cup X'',\] 
    where $[a,b]$ denotes the usual segment in the plane between $a$ and $b$, and 
    \[
    \chi_1\colon \begin{cases}
        \begin{aligned}
            X'&\,\mapsto 0,\\
            X'' &\,\mapsto 1,
        \end{aligned}
    \end{cases}\quad\text{and}\quad
    \chi_2\colon \begin{cases}
        \begin{aligned}
            X'&\,\mapsto 1,\\
            X'' &\,\mapsto 0,
        \end{aligned}
    \end{cases}
    \]
    (see Figure \ref{fig:chromatic_invariants}).
    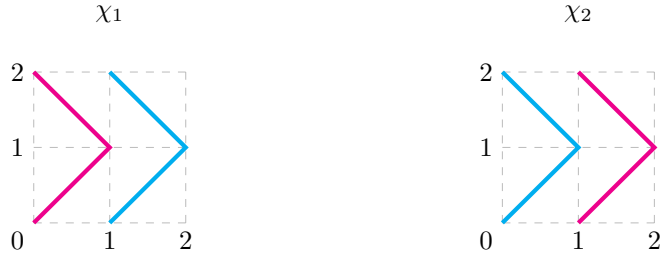
\begin{figure}[h!]
        \centering
        \begin{subfigure}{0.38\textwidth}
            \centering
            \begin{tikzpicture}
                \draw[lightgray, dashed] (1,0)--(1,2) (2,0)--(2,2) (0,1)--(2,1) (0,0)--(0,2) (0,0)--(2,0) (0,2)--(2,2);
                \draw (0,0) node[below left]{$0$};
                \draw (1,0) node[below]{$1$};
                \draw (2,0) node[below]{$2$};
                \draw (0,1) node[left]{$1$};
                \draw (0,2) node[left]{$2$};
                \draw[magenta,ultra thick] (0,0)--(1,1)--(0,2);
                \draw[cyan,ultra thick] (1,0)--(2,1)--(1,2);
                \draw (1,3) node[below]{$\chi_1$};
            \end{tikzpicture}
        \end{subfigure}
        \begin{subfigure}{0.38\textwidth}
            \centering
            \begin{tikzpicture}
                \draw[lightgray, dashed] (1,0)--(1,2) (2,0)--(2,2) (0,1)--(2,1) (0,0)--(0,2) (0,0)--(2,0) (0,2)--(2,2);
                \draw (0,0) node[below left]{$0$};
                \draw (1,0) node[below]{$1$};
                \draw (2,0) node[below]{$2$};
                \draw (0,1) node[left]{$1$};
                \draw (0,2) node[left]{$2$};
                \draw[cyan,ultra thick] (0,0)--(1,1)--(0,2);
                \draw[magenta,ultra thick] (1,0)--(2,1)--(1,2);
                \draw (1,3) node[below]{$\chi_2$};
            \end{tikzpicture}
        \end{subfigure}
        \caption{\footnotesize A representation of the chromatic metric spaces $\chi_1$ and $\chi_2$ defined in Example~\ref{ex:chromatic_invariants}(i).
        }\label{fig:chromatic_invariants}
    \end{figure}

    Consider the discrete constraint set $C_D$. In this case, it is easy to see that $\chi_1^{-1}(\sigma)$ is isometric to $\chi_2^{-1}(\sigma)$ for every $\sigma\in C_D$, and so the usual Gromov-Hausdorff distance does not help providing lower bounds to the $C_D$-constrained Gromov-Hausdorff distance. 
    
    We intend to apply Proposition \ref{prop:chromatic_invariants}. Consider two singletons $\{0\}$ and $\{1\}$, which form an open partition (more formally, $\{0\}$ and $\N\setminus\{0\}$ do that). It is easy to see that $\rad_{\chi_1}^{\{0\},\{1\}}=1$ (it is enough to consider the point $(1,1)$). However, $\rad_{\chi_2}^{\{0\},\{1\}}=3$. Therefore, 
    \[
    \GH^{C_D}(\chi_1,\chi_2)\geq\frac{1}{2}\lvert\rad_{\chi_1}^{\{0\},\{1\}}-\rad_{\chi_2}^{\{0\},\{1\}}\rvert= 1.
    \]

    Consider the intuitive map $f\colon \chi_1 \to \chi_2$ swapping $X'$ with $X''$ and its inverse $f^{-1}$. They satisfy 
    \[
    \max\{\dis f,\dis(f^{-1}),\codis(f,f^{-1})\}=2,
    \]
    and so $\GH^{C_D}(\chi_1,\chi_2)=1$.
\item We provide a chromatic metric space $\chi_1$, a sequence $\{\chi_\varepsilon\}_{0<\varepsilon\leq1}$ of chromatic metric spaces, and a constraint set $\N\in C\subseteq\mathcal P(\N)$ such that
\begin{compactenum}[(a)]
\item $\GH(\chi_1^{-1}(\sigma),\chi_2^{-1}(\sigma))\leq\varepsilon$ for every $\sigma\in C$, and
\item there are two elements $\sigma$ and $\tau$ of an open partition of $(\N,\Tau_C)$ such that $\GH^C(\chi_1,\chi_\varepsilon)=\lvert\dist^{\sigma,\tau}_{\chi_1}-\dist^{\sigma,\tau}_{\chi_\varepsilon}\rvert/2$.
\end{compactenum}
Let us define the following chromatic metric spaces: for $0<\varepsilon\leq 1$, $(\N\cup\{-\varepsilon\},\chi_\varepsilon)$, where 
\[
\chi_\varepsilon\colon\begin{cases}
    \begin{aligned}
        \N &\mapsto 0,\\
        -\varepsilon&\mapsto 1,
    \end{aligned}
\end{cases}
\]
and the metric is induced by the usual distance on the real line (see Figure \ref{fig:dist}).
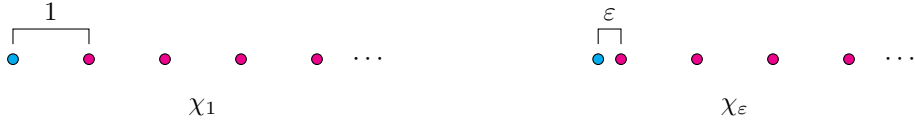
\begin{figure}[h!]
    \centering
    \begin{tikzpicture}
        \fill[magenta] (0,0) circle (2pt) (1,0) circle (2pt) (2,0) circle (2pt) (3,0) circle (2pt);
        \fill[cyan] (-1,0) circle (2pt);
        \draw (0,0) circle (2pt) (1,0) circle (2pt) (2,0) circle (2pt) (3,0) circle (2pt);
        \draw (-1,0) circle (2pt);
        \draw (3.7,0) node{$\cdots$};
        \draw (-1,0.2)--(-1,0.4)--(0,0.4)node[pos=0.5,above]{$1$}--(0,0.2);
        \draw (1.5,-0.4) node[below]{$\chi_1$};
            
        \fill[magenta] (7,0) circle (2pt) (8,0) circle (2pt) (9,0) circle (2pt) (10,0) circle (2pt);
        \fill[cyan] (6.7,0) circle (2pt);
        \draw (7,0) circle (2pt) (8,0) circle (2pt) (9,0) circle (2pt) (10,0) circle (2pt);
        \draw (6.7,0) circle (2pt);
        \draw (10.7,0) node{$\cdots$};
        \draw (6.7,0.2)--(6.7,0.4)--(7,0.4)node[pos=0.5,above]{$\varepsilon$}--(7,0.2);
        \draw (8.5,-0.4) node[below]{$\chi_\varepsilon$};
    \end{tikzpicture}
    \caption{\footnotesize A representation of the chromatic metric pairs $\chi_1$ and $\chi_\varepsilon$ defined in Example  \ref{ex:chromatic_invariants}(ii).}
    \label{fig:dist}
\end{figure}
Let us estimate the $C_D$-constrained Gromov-Hausdorff distance between $\chi_1$ and $\chi_\varepsilon$. Clearly, $\chi_1^{-1}(0)$ and $\chi_\varepsilon^{-1}(0)$, and $\chi_1^{-1}(1)$ and $\chi_\varepsilon^{-1}(1)$ are pairwise isometric. Define two maps,  $f_\varepsilon\colon\N\cup\{-1\}\to\N\cup\{-\varepsilon\}$ and $g_\varepsilon\colon\N\cup\{-\varepsilon\}\to\N\cup\{-\varepsilon\}$ as follows:
\[
f_\varepsilon\colon n\mapsto n+1,\text{ for $n\in\N\cup\{-1\}$, and }g_\varepsilon\colon\begin{cases}
    \begin{aligned}
        -\varepsilon&\mapsto -1,\\
        n&\mapsto n-1,\text{ for $n\in\N$}
    \end{aligned}
\end{cases}
\]
(see Figure \ref{fig:dist_2}). This pair of maps shows that 
\[
2\GH(\chi_1^{-1}(\N),\chi_\varepsilon^{-1}(\N))\leq\max\{\dis f_\varepsilon,\dis g_\varepsilon,\codis(f_\varepsilon,g_\varepsilon)\}=\varepsilon.
\]
However, we immediately see that neither map is $C_D$-constrained.

A $C_D$-constrained bijection $h\colon\chi_1\to\chi_\varepsilon$ can be defined as $h(n)=n$ for every $n\in\N$, and $h(-1)=\varepsilon$. Note that also $h^{-1}$ is $C_D$-constrained (see again Figure \ref{fig:dist_2}). Then,
\[
2\GH^{C_D}(\chi_1,\chi_\varepsilon)\leq\max\{\dis h,\dis h^{-1},\codis(h,h^{-1})\}=1-\varepsilon.
\]

\begin{figure}
[h!]
    \centering
    \begin{tikzpicture}
        \draw[->,dashed,thick] (-1,0) to [out=-40,in=90] (0,-2);
        \draw[->,dashed,thick] (0,0)to [out=-40,in=90](1,-2); 
        \draw[->,dashed,thick] (1,0)to [out=-40,in=90](2,-2);
        \draw[->,dashed,thick] (2,0)to [out=-40,in=90](3,-2);
        \draw[->,dotted,thick] (0,-2) to [out=130,in=-90](-1,0) ;
        \draw[->,dotted,thick] (1,-2) to [out=130,in=-90] (0,0); 
        \draw[->,dotted,thick] (2,-2)to [out=130,in=-90](1,0);
        \draw[->,dotted,thick] (3,-2)to [out=130,in=-90](2,0);
        \draw[->,dotted,thick] (-0.3,-2)to [out=120,in=-90](-1,0);
    
        \fill[magenta] (0,0) circle (2pt) (1,0) circle (2pt) (2,0) circle (2pt) (3,0) circle (2pt);
        \fill[cyan] (-1,0) circle (2pt);
        \draw (0,0) circle (2pt) (1,0) circle (2pt) (2,0) circle (2pt) (3,0) circle (2pt);
        \draw (-1,0) circle (2pt);
        \draw (3.7,0) node{$\cdots$};
        \draw (-1,0.2)--(-1,0.4)--(0,0.4)node[pos=0.5,above]{$1$}--(0,0.2);
        \draw (1.5,0.4) node[above]{$\chi_1$};
        
        \fill[magenta] (0,-2) circle (2pt) (1,-2) circle (2pt) (2,-2) circle (2pt) (3,-2) circle (2pt);
        \fill[cyan] (-0.3,-2) circle (2pt);
        \draw (0,-2) circle (2pt) (1,-2) circle (2pt) (2,-2) circle (2pt) (3,-2) circle (2pt);
        \draw (-0.3,-2) circle (2pt);
        \draw (3.7,-2) node{$\cdots$};
        \draw (-0.3,-2.2)--(-0.3,-2.4)--(0,-2.4)node[pos=0.5,below]{$\varepsilon$}--(0,-2.2);
        \draw (1.5,-2.4) node[below]{$\chi_\varepsilon$};

        \draw[->] (7,0)--(7,-2); 
        \draw[->] (8,0)--(8,-2);
        \draw[->] (9,0)--(9,-2);
        \draw[<-] (10,-2)--(10,0);
        \draw[<-] (6.7,-2)--(6,0);
    
        \fill[magenta] (7,0) circle (2pt) (8,0) circle (2pt) (9,0) circle (2pt) (10,0) circle (2pt);
        \fill[cyan] (6,0) circle (2pt);
        \draw (7,0) circle (2pt) (8,0) circle (2pt) (9,0) circle (2pt) (10,0) circle (2pt);
        \draw (6,0) circle (2pt);
        \draw (10.7,0) node{$\cdots$};
        \draw (6,0.2)--(6,0.4)--(7,0.4)node[pos=0.5,above]{$1$}--(7,0.2);
        \draw (8.5,0.4) node[above]{$\chi_1$};
        
        \fill[magenta] (7,-2) circle (2pt) (8,-2) circle (2pt) (9,-2) circle (2pt) (10,-2) circle (2pt);
        \fill[cyan] (6.7,-2) circle (2pt);
        \draw (7,-2) circle (2pt) (8,-2) circle (2pt) (9,-2) circle (2pt) (10,-2) circle (2pt);
        \draw (6.7,-2) circle (2pt);
        \draw (10.7,-2) node{$\cdots$};
        \draw (6.7,-2.2)--(6.7,-2.4)--(7,-2.4)node[pos=0.5,below]{$\varepsilon$}--(7,-2.2);
        \draw (8.5,-2.4) node[below]{$\chi_\varepsilon$};
    \end{tikzpicture}
    \caption{\footnotesize A representation of the pairs of maps used to show upper bounds in Example  \ref{ex:chromatic_invariants}(ii). On the left-hand side, we represent the maps $f_\varepsilon$ (dashed) pointing downwards and $g_\varepsilon$ (dotted) pointing upwards. On the right-hand side, we show the $C_D$-constrained bijection $h$.}
    \label{fig:dist_2}
\end{figure}
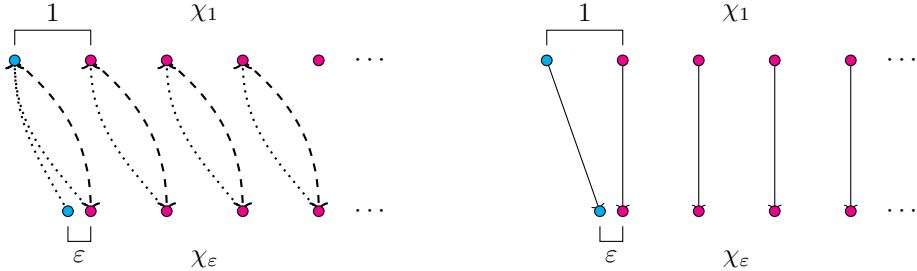

We intend to show that $2\GH^{C_D}(\chi_1,\chi_\varepsilon)\geq 1-\varepsilon$, thus achieving the equality. Consider, as in item (i), the two singletons $\{0\}$ and $\{1\}$ which form an open partition. Then, $\dist_{\chi_1}^{\{0\},\{1\}}=1$ and $\dist_{\chi_\varepsilon}^{\{0\},\{1\}}=\varepsilon$. Hence, Proposition \ref{prop:chromatic_invariants} implies that
\[
2\GH^{C_D}(\chi_1,\chi_\varepsilon)\geq\lvert\dist_{\chi_1}^{\{0\},\{1\}}-\dist_{\chi_\varepsilon}^{\{0\},\{1\}}\rvert=1-\varepsilon. 
\]

The reader may notice that the same lower bound could be achieved by considering the distance between $\rad_{\chi_1}^{\{0\},\{1\}}$ and $\rad_{\chi_2}^{\{0\},\{1\}}$. However, the example can be modified to make this second approach not fruitful (see Figure \ref{fig:dist_3}).
\begin{figure}
[h!]
    \centering
    \begin{tikzpicture}
        \fill[magenta] (0,0) circle (2pt) (1,0) circle (2pt) (2,0) circle (2pt) (3,0) circle (2pt);
        \fill[cyan] (-1,0) circle (2pt) (-2,0) circle (2pt) (-3,0) circle (2pt) (-4,0) circle (2pt);
        \draw (0,0) circle (2pt) (1,0) circle (2pt) (2,0) circle (2pt) (3,0) circle (2pt);
        \draw (-1,0) circle (2pt) (-2,0) circle (2pt) (-3,0) circle (2pt) (-4,0) circle (2pt);
        \draw (3.7,0) node{$\cdots$};
        \draw (-4.7,0) node{$\cdots$};
        \draw (-1,-0.2)--(-1,-0.4)--(0,-0.4)node[pos=0.5,below]{$1$}--(0,-0.2);
        \draw (-0.5,0.4) node[above]{$\chi'_1$};
        
        \fill[magenta] (0,-2) circle (2pt) (1,-2) circle (2pt) (2,-2) circle (2pt) (3,-2) circle (2pt);
        \fill[cyan] (-0.3,-2) circle (2pt) (-1.3,-2) circle (2pt) (-2.3,-2) circle (2pt) (-3.3,-2) circle (2pt);
        \draw (0,-2) circle (2pt) (1,-2) circle (2pt) (2,-2) circle (2pt) (3,-2) circle (2pt);
        \draw (-0.3,-2) circle (2pt) (-1.3,-2) circle (2pt) (-2.3,-2) circle (2pt) (-3.3,-2) circle (2pt);
        \draw (3.7,-2) node{$\cdots$};        
        \draw (-4,-2) node{$\cdots$};
        \draw (-0.3,-1.8)--(-0.3,-1.6)--(0,-1.6)node[pos=0.5,above]{$\varepsilon$}--(0,-1.8);
        \draw (-0.5,-2.4) node[below]{$\chi'_\varepsilon$};
    \end{tikzpicture}
    \caption{\footnotesize A modification of the chromatic metric spaces provided in Example \ref{ex:chromatic_invariants}(ii). The two underlying metric spaces of $\chi_1'$ and $\chi_\varepsilon'$ are the subspaces $\Z$ and $\N\cup(-\N-\varepsilon)=\N\cup\{-n-\varepsilon\mid n\in\N\}$ of the real line, respectively. The coloring functions satisfy $\chi_1'\colon\N\mapsto 0$, $\chi_1'\colon\Z\setminus\N\mapsto 1$, $\chi_\varepsilon'\colon\N\mapsto 0$ and $\chi_\varepsilon'\colon-\N-\varepsilon\mapsto 1$. As in Example \ref{ex:chromatic_invariants}(ii), $\GH^{C_D}(\chi_1',\chi_\varepsilon')=\tfrac{1-\varepsilon}{2}$, which can be shown using an immediate adaptation of the map $h$ and comparing $\dist_{\chi_1'}^{\{0\},\{1\}}=1$ and $\dist_{\chi_\varepsilon'}^{\{0\},\{1\}}=\varepsilon$. However, $\rad_{\chi_1'}^{\{0\},\{1\}}=\rad_{\chi_\varepsilon'}^{\{0\},\{1\}}=\infty$, and so the lower bound cannot be alternatively provided.}
    \label{fig:dist_3}
\end{figure}
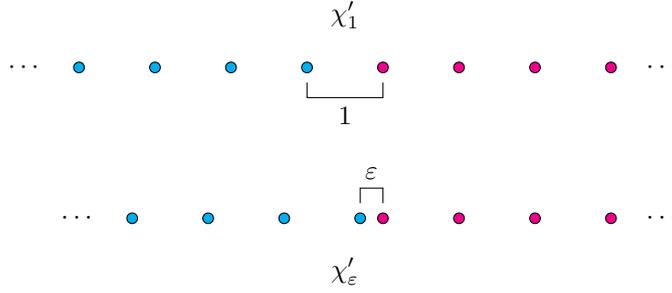

\end{compactenum}
\end{example}

\subsection{Characterizations of the $C$-constrained Gromov-Hausdorff distance}\label{sub:characterisations}

Let $\sigma\subseteq\N$ be a constraint. Given two chromatic metric pairs $\MP{1}$ and $\MP{2}$, we say that a correspondence $R\in\mathcal R(A_1,A_2)$ is {\em $\sigma$-constrained} if $R|_{\chi_1^{-1}(\sigma)\times\chi_2^{-1}(\sigma)}$ is a correspondence. If $C$ is a family of subsets of $\N$ containing $\N$, we say that a correspondence $R\in\mathcal R(A_1,A_2)$ is {\em $C$-constrained} if $R$ is $\sigma$-constrained for every $\sigma\in C$. Let us denote by $\mathcal R^C(\MP{1},\MP{2})$ the set of all $C$-constrained correspondences between $\MP{1}$ and $\MP{2}$.

\begin{theorem}\label{thm:characterization_via_correspondences}
	Let $\extMP{1}$ and $\extMP{2}$ be two chromatic metric pairs. Let $\N\in C\subseteq\mathcal P(\N)$ be a constraint set.
    Then, for every constraint set $\Sigma_C\subseteq C'\subseteq\Tau_C$ (see Fact \ref{fact:C_n_and_sigma_n} for the definition of $\Sigma_C$),
	\[
	\GH^C(\MP{1},\MP{2})=\GH^{C'}(\MP{1},\MP{2})=\frac{1}{2}\inf_{R\in\mathcal R^{C'}(\MP{1},\MP{2})}\dis R.
	\]
\end{theorem}
\begin{proof}
    From the assumption on $C'$, $\Tau_C=\Tau_{C'}$, and thus the first equality follows as in 
    Remark~\ref{rem:C-constrained_GH_and_other_GH}(v). 
    
    Let us now discuss the second one. The proof is essentially the same of that of Theorem~\ref{thm:Kalton_Ostrovskii}. Given two $C'$-constrained maps $f\colon A_1\to A_2$ and $g\colon A_2\to A_1$, we need to show that $R=\gr(f)\cup(\gr(g))^{-1}$ is a $C'$-constrained correspondence. This fact is trivial since, for every $\sigma\in C'$, if $x\in\chi_1^{-1}(\sigma)$, then $f(x)\in\chi_2^{-1}(\sigma)$ and $(x,f(x))\in R$. A similar conclusion can be shown for every point in $\chi_2^{-1}(\sigma)$. Thus, $R$ is $\sigma$-constrained. 
	
	Let us consider the converse construction. 
    Fix a $C'$-constrained correspondence $R$. In particular, $R$ is $\Sigma_C$-constrained. For every $x\in A_1\setminus X_1$, pick an arbitrary point $y\in A_2$ such that $(x,y)\in R$ and set $f(x)=y$. Consider now a point $x$ belonging to $X_1$. Set $n=\chi_1(x)$. Choose any point $y\in X_2$ such that $(x,y)\in R|_{\chi_1^{-1}(\sigma_n)\times\chi_2^{-1}(\sigma_n)}$, which exists since $R|_{\chi_1^{-1}(\sigma_n)\times\chi_2^{-1}(\sigma_n)}$ is a correspondence. Define $f(x)=y$. Then, the so-defined map $f$ is $\Sigma_C$-constrained and so $C'$-constrained. Similarly, we can define a map $g$ in the opposite direction.
\end{proof}

Let us discuss the further assumption on the constraint set added in the statement Theorem \ref{thm:characterization_via_correspondences}.

\begin{remark}\label{rem:why_intersection?}
	The reader may wonder why we used $C'$ rather than the original constraint set $C$ in Theorem~\ref{thm:characterization_via_correspondences}; something we did not do in Definition~\ref{def:C-constrained_GH}.
	
	First of all, we note that, according to Remark~\ref{rem:C-constrained_GH_and_other_GH}(v), it is not a restriction on constraint sets we can use. Indeed, $\Tau_C=\Tau_{C'}$ (Proposition~\ref{prop:C-constrained_iff_TC-constrained}) and so the two constraint sets have the same strength.
	
	Secondly, if we consider a constraint set $C'$ that has the same strength as $C$, but does not contain $\Sigma_C$ (e.g., potentially even $C$ itself), 
    
    we can construct examples of pairs of chromatic metric pairs where there are $C'$-constrained correspondences, but no $C$-constrained map can be defined. Take two chromatic metric pairs 
	\[
	(\{a,b,c\},\chi_1\colon a\mapsto 0,\,b\mapsto 1,\, c\mapsto 2)\quad\text{and}\quad (\{x,y\},\chi_2\colon x\mapsto 0,\, y\mapsto 2)
	\]
    (we do not specify the distances since they do not play any role). Let $C=\{\{0,1\},\{1,2\}\}$. Then, $R=\{(a,x),(b,x),(b,y),(c,y)\}$ is a $C$-constrained correspondence. However, there exists no $C$-constrained map $f\colon\{a,b,c\}\to\{x,y\}$ since, wherever $b$ is sent, it violates one of the two constraints. We represent the situation in Figure \ref{fig:C_closed_under_intersection}.
	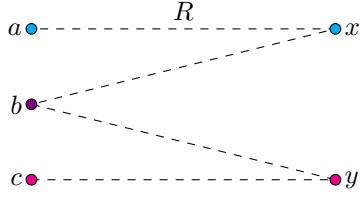
\begin{figure}[h!]
		\centering
		\begin{tikzpicture}
			\draw (0,0) node[left]{$c$};
			\draw (0,1) node[left]{$b$};
			\draw (0,2) node[left]{$a$};
			\draw (4,0) node[right]{$y$};
			\draw (4,2) node[right]{$x$};
			\draw (2,2) node[above]{$R$};
			\draw[dashed] (0,0)--(4,0)--(0,1)--(4,2)--(0,2); 
			\fill[magenta] (0,0) circle (2pt) (4,0) circle (2pt);
			\fill[violet] (0,1) circle (2pt);
			\fill[cyan] (0,2) circle (2pt) (4,2) circle (2pt);
            \draw (0,0) circle (2pt) (4,0) circle (2pt);
			\draw (0,1) circle (2pt);
			\draw (0,2) circle (2pt) (4,2) circle (2pt);
		\end{tikzpicture}
	\caption{\footnotesize A representation of the example constructed in Remark \ref{rem:why_intersection?}. The three colors $0,1,2$ are represented in cyan, violet and magenta, respectively. Since both $R|_{\{a,b\}\times\{x\}}$ and $R|_{\{b,c\}\times\{y\}}$ are correspondences, $R$ is a $C$-constrained correspondence. However, if a map sends the violet-colored point $b$ to $x$, it violates the constraint $\{\text{violet, magenta}\}$. Alternatively, if it is sent to $y$, it violates $\{\text{violet, cyan}\}$. }\label{fig:C_closed_under_intersection}
	\end{figure}
\end{remark}

\begin{definition}
	Let $\MP{}$ be a chromatic metric pair and $Y$ and $Z$ be two subspaces thereof. Let $\N\in C\subseteq\mathcal P(\N)$ be a constraint set. Then, their {\em $C$-constrained Hausdorff distance} is defined as
	\[
	d_H^C(Y,Z)=\newinf_{R\in\mathcal R^C(Y,Z)}\sup_{(y,z)\in R}d(y,z).
	\]
\end{definition}

There is a fundamental difference between the chromatic versions of the Gromov-Hausdorff and the Hausdorff distance. In the latter, the embeddings of the subspaces are already given. Therefore, measuring the distances of individual subspaces given by various sets of colors $\sigma$ can happen independently of each other, and we get the following relation.

\begin{proposition}\label{prop:d_H_rewritten}
Let $(A, \chi\colon X\rightarrow \N)$ be a chromatic metric pair, $Y, Z\subseteq X$ be two subspaces, $\N \in C\subseteq \mathcal{P}(\N)$. For each $\sigma\in C$, let $Y_\sigma=\chi^{-1}(\sigma)\cap Y$ and $Z_\sigma=\chi^{-1}(\sigma)\cap Z$.
    Then \[
        \max\{d_H(Y,Z),\sup_{\sigma\in C} d_H(Y_\sigma, Z_\sigma)\} = d_H^C(Y, Z).
    \]
\end{proposition}

\begin{proof}
    By definition, if $R\in \mathcal R^C(Y,Z)$, then $R\cap (Y_\sigma \times Z_\sigma)$ is a correspondence between $Y_\sigma$ and $Z_\sigma$. This implies $d_H(Y_\sigma, Z_\sigma) \leq d_H^C(Y, Z)$, and then also $\sup_{\sigma\in C} d_H(Y_\sigma, Z_\sigma) \leq d_H^C(Y, Z)$. Clearly also $d_H(Y, Z)\leq d_H^C(Y,Z)$, as $\mathcal R^C(Y,Z)\subseteq \mathcal R(Y,Z)$.

    For the other direction, consider a correspondence $R_\sigma \in \mathcal{R}(Y_\sigma, Z_\sigma)$ for each $\sigma \in C$, and a correspondence $R'\in\mathcal R(Y,Z)$. Let $R=R'\cup\big(\bigcup_{\sigma\in C}R_\sigma\big)$.
    Then, $R\in \mathcal R^C(Y,Z)$, because $R_\sigma \subseteq R\cap(Y_\sigma \times Z_\sigma)$ for each $\sigma\in C$.  Clearly, \[
        \sup_{\sigma\in C}\sup_{(y, z) \in R_\sigma} d(y, z)
        = \sup_{(y, z)\in R} d(y, z)
        \geq d_H^C(Y, Z).
    \]
    Passing to sequences of correspondences converging to the infima from the definition of $d_H(Y_\sigma, Z_\sigma)$ on the left-hand side, we get the desired inequality $\sup_{\sigma\in C} d_H(Y_\sigma, Z_\sigma) \geq d_H^C(Y, Z)$.
\end{proof}

Let $\MP{1}$ and $\MP{2}$ be two chromatic metric pairs. For an admissible metric $d\in\mathcal D(A_1,A_2)$ (see Theorem \ref{thm:GH_via_embed} for the definition of an admissible metric), we can endow $(A_1\sqcup A_2,d)$ with a chromatic metric pair structure $(A', \chi'\colon X' \to \N)$ by setting $A'=(A_1\sqcup A_2,d)$, $X'=i_1(X_1)\sqcup i_2(X_2)$, where $i_k\colon A_k\to A_1\sqcup A_2$ for $k=1,2$ are the canonical inclusion, and $\chi'(i_k(x))=\chi_k(x)$ for $k=1,2$ and $x\in X_k$. With this additional structure on $A'$, the inclusions $i_k\colon \MP{k}\to (A',\chi')$ for $k=1,2$ are $C$-constrained isomorphic embeddings. We write $i_k(\MP{k})$ to denote the corresponding subspace of $(A',\chi')$ to differentiate it from the metric subspace $i_k(A_k)$ of $A'$.

\begin{theorem}\label{thm:C-GH_via_embeddings}
	Let $\N\in C\subseteq\mathcal P(\N)$ be a constraint set and $\Sigma_C\subseteq C'\subseteq\Tau_C$. Then, for every pair of chromatic metric pairs $\MP{1}$ and $\MP{2}$,
    \[
    \GH^C(\MP{1},\MP{2})=\GH^{C'}(\MP{1},\MP{2})=\inf_{d\in\mathcal D(A_1,A_2)}d_H^{C'}(i_1(\MP{1}),i_2(\MP{2})).
    \]
\end{theorem}
\begin{proof}
    The first equality is immediate. As for the second one, the proof is an adaptation of that of Theorem~\ref{thm:GH_via_embed}. Indeed, it is easy to see that the correspondences constructed as in the proof of Theorem~\ref{thm:GH_via_embed} have the desired additional properties.
\end{proof}

\subsubsection{Comparison with other existing notions}\label{subsub:related_notions}

Let us specialize Theorem \ref{thm:C-GH_via_embeddings} in the setting of metric pairs. We have discussed in Remark \ref{rem:C-constrained_GH_and_other_GH}(iv) how the $C$-constrained Gromov-Hausdorff distance canonically induces a Gromov-Hausdorff distance between metric pairs.

\begin{corollary}\label{coro:GH_metric_pairs}
    Let $(A_1,X_1)$ and $(A_2,X_2)$ be two metric pairs. Then,
    \[
    \GH((A_1,X_1),(A_2,X_2))=\inf_{d\in\mathcal{D}(A_1,A_2)}\max\{d_H(i_1(A_1),i_2(A_2)),d_H(i_1(X_1),i_2(X_2))\}
    \]
\end{corollary}
\begin{proof}
    It follows from Theorem \ref{thm:C-GH_via_embeddings} and Proposition \ref{prop:d_H_rewritten}.
\end{proof}

The above result implies that our Gromov-Hausdorff distance between metric pairs coincides with that defined in \cite{TorGon}. In \cite{GomChe}, the authors provide a different, but related Gromov-Hausdorff distance notion between metric pairs. Let $(A_1,X_1)$ and $(A_2,X_2)$ be two metric pairs. The authors define their Gromov-Hausdorff distance as the value
\[
    \widetilde{\GH}((A_1,X_1),(A_2,X_2))=\inf_{d\in\mathcal D(A_1,A_2)
    }d_H(i_1(A_1),i_2(A_2))+d_H(i_1(X_1),i_2(X_2).
\]
The following result, which is an immediate application of Corollary \ref{coro:GH_metric_pairs}, compares it with our notion.

\begin{corollary}\label{coro:GH_metric_pairs}
	For every two metric pairs $(A_1,X_1)$ and $(A_2,X_2)$,
	\begin{equation*}
		\GH((A_1,X_1),(A_2,X_2))\leq\widetilde{\GH}((A_1,X_1),(A_2,X_2))\leq 2\GH((A_1,X_1),(A_2,X_2)).
	\end{equation*}
\end{corollary}

Since the two metrics $\GH$ and $\widetilde{\GH}$ are equivalent (i.e., the identity map is a bi-Lipschitz equivalence), the topological statements obtained in \cite{GomChe} for $\widetilde{\GH}$ can be also applied to the distance presented in this work.

\subsection{Characterization of spaces at zero $C$-constrained Gromov-Hausdorff distance}\label{sub:inf_is_min}

We want to prove the existence of the distortion-minimizing correspondence. From such result, a characterization of those spaces that have $C$-constrained Gromov-Hausdorff distance $0$ can be immediately provided. First, we need a couple of technical lemmas.

\begin{lemma}\label{lemma:closure_same_distortion}
    Let $R$ be a relation between the two metric spaces $X$ and $Y$. Then, $\dis\overline R=\dis R$ where $\overline R$ is the closure of $R$ in the product space $X\times Y$ endowed with the product topology.
\end{lemma}
\begin{proof}
    Clearly, $\dis R\leq\dis\overline R$. Let $(\overline x_i,\overline y_i)\in\overline R$ for $i=1,2$, and let $(x_i^n,y_i^n)$ be two sequences in $R$ converging to $(\overline x_i,\overline y_i)$. For every $\varepsilon>0$, $d_X(x_i^n,\overline x_i)<\varepsilon$ and $d_Y(y_i^n,\overline y_n)<\varepsilon$ eventually. Therefore, in the chain of inequalities
    \[
        \begin{aligned}
        d_X(\overline x_1,\overline x_2)-d_Y(\overline y_1,\overline y_2)&\,\leq d_X(\overline x_1,x_1^n)+d_X(x_1^n,x_2^n)+d_X(x_2^n,\overline x_2)-\\
        &\hspace{3cm}-d_Y(y_1^n,y_2^n)+d_Y(\overline y_1,y_1^n)+d_Y(y_2^n,\overline y_2)\leq\\
        &\,\leq \dis R+d_X(\overline x_1,x_1^n)+d_X(x_2^n,\overline x_2)+d_Y(\overline y_1,y_1^n)+d_Y(y_2^n,\overline y_2),
        \end{aligned}
    \]
    the latter term is eventually bounded by $\dis R+4\varepsilon$. Since $\varepsilon$ can be arbitrarily taken, $\dis\overline R\leq\dis R$.
\end{proof}

A metric pair $(A,X)$ is {\em compact} if $A$ is compact and $X$ is closed in $A$.

\begin{lemma}\label{lemma:pair_correspondences_converge_to_pair_correspondences}
    Let $(A_1,X_1)$ and $(A_2,X_2)$ be two compact metric pairs. Let $\{R_n\}_n$ be a sequence in $\mathcal R(A_1,A_2)$ such that 
    \begin{compactenum}[(i)]
        \item $R_n|_{X_1\times X_2}$ are correspondences,
        \item $R_n$ are closed in $A_1\times A_2$ endowed with the supremum metric, and 
        \item $R_n\xrightarrow{d_H}R$ where $R$ is a closed correspondence.
    \end{compactenum}
    Then, $R|_{X_1\times X_2}$ is a correspondence.
\end{lemma}
\begin{proof}
    Let $x\in X_1$. For every $n\in\N$, there is $y_n\in X_2$ such that $(x,y_n)\in R_n$. Since $X_2$ is compact, $\{y_n\}_n$ (or possibly a subsequence thereof) converges to a point $y\in X_2$. In particular, $(x,y_n)\to(x,y)$, We claim that $(x,y)\in R$. Indeed, according to \cite[Propositions 5.18, 5.20]{Tuz}, 
    \[
    R=\overline R=\{z\in A_1\times A_2\mid \exists\{z_n\}_n,\, z_n\in R_n:\, z_n\to z\}
    \]
    and so $(x,y)\in R$. Similarly, we can show that, for every point $y\in X_2$, there exists $x\in X_1$ such that $(x,y)\in R$.
\end{proof}

Let $\N\in C\subseteq\mathcal P(\N)$. We say that a chromatic metric pair $\MP{}$ is {\em $C$-compact} if $A$ is compact, and every $\chi^{-1}(\sigma)\subseteq A$, for $\sigma\in C$, is closed. 

\begin{remark}\label{rem:C_and_Sigma_C_compact}
Let $\N\in C\subseteq\mathcal P(\N)$. A $C$-compact chromatic metric pair $\MP{}$ is automatically $(\Sigma_C\cup\{\N\})$-compact. Indeed, for every $n\in\N$, $\chi^{-1}(\sigma_n)=\bigcap_{\sigma\in C_n}\chi^{-1}(\sigma)$, which is closed provided that so is each $\chi^{-1}(\sigma)$.

The converse implication does not hold in general. Consider the discrete constraint set $C_D$. Then, $\sigma_n=\{n\}$ for every $n\in\N$. Define the chromatic metric pair $([0,1],\chi)$ where
\[
\chi\colon\begin{cases}
    \begin{aligned}
        0\,&\mapsto 0 &\\
        \frac{1}{n}\,&\mapsto n &\text{for every $n\in\N\setminus\{0\}$.}
    \end{aligned}
\end{cases}
\]
Then, $\chi^{-1}(\N)$ is closed since it is the union of a convergent sequence and its limit point, and $\chi^{-1}(\sigma_n)$ is a singleton for every $n\in\N$. Hence, $([0,1], \chi)$ is $(\Sigma_{C_D}\cup\{\N\})$-compact. However, $\chi^{-1}(\N\setminus\{0\})=\{1/n\mid n\in\N\setminus\{0\}\}$ is not closed, which shows that $([0,1],\chi)$ is not $C$-compact.
\end{remark}

\begin{proposition}\label{prop:correspondence_minimizing_distortion_metric_pairs}
    Let $\N\in C\subseteq\mathcal P(\N)$ be a constraint set containing $\Sigma_C$
    , and $\MP{1}$ and $\MP{2}$ be two $C$-compact chromatic metric pairs. Suppose that $\GH^C(\MP{1},\MP{2})<\infty$. Then, there is $R\in \mathcal R^C(\MP{1},\MP{2})$ such that $\GH^C(\MP{1},\MP{2})=\frac{1}{2}\dis R$. 
\end{proposition}
\begin{proof}
    According to the hypotheses, we can characterize the $C$-constrained Gromov-Hausdorff distance using $C$-constrained correspondences (Theorem \ref{thm:characterization_via_correspondences}). Let $R_n\in\mathcal R^C(\MP{1},\MP{2})$ be a sequence of $C$-constrained correspondences such that 
    \[
        2\GH^C(\MP{1},\MP{2})\leq\dis R_n\leq2\GH^C(\MP{1},\MP{2})+1/n. 
    \]
    According to Lemma \ref{lemma:closure_same_distortion}, we can assume that each $R_n$ is closed. Note that $A_1\times A_2$ is compact, and the space of all its compact subsets endowed with the Hausdorff distance is also compact by (Blaschke's Theorem, see \cite{BurBurIva}). According to \cite{ChoMem_geo}, the limit $R$ of (a subsequence of) $\{R_n\}_n$ is a correspondence between $A_1$ and $A_2$ and $\dis R=\lim\dis R_n=2\GH^C(\MP{1},\MP{2})$. Consider an arbitrary $\sigma\in C$. Both $\chi_1^{-1}(\sigma)$ and $\chi_2^{-1}(\sigma)$ are compact by assumption. By Lemma \ref{lemma:pair_correspondences_converge_to_pair_correspondences}, $R|_{\chi_1^{-1}(\sigma)\times\chi_2^{-1}(\sigma)}$ is a correspondence and so $R$ is $\sigma$-constrained.
\end{proof}

As a corollary, we obtain a characterization of those chromatic metric pairs at distance~$0$.

\begin{corollary}\label{coro:0_distance_chromatic}
    Let $\N\in C\subseteq\mathcal P(\N)$ be a constraint set containing $\Sigma_C$. Suppose that $\MP{1}$ and $\MP{2}$ are two $(\Sigma_C\cup\{\N\})$-compact chromatic metric pairs. Then, the following properties are equivalent:
    \begin{compactenum}[(i)]
        \item $\GH^C(\MP{1},\MP{2})=0$;
        \item $\GH^{\Sigma_C\cup\{\N\}}(\MP{1},\MP{2})=0$;
        \item $\MP{1}$ and $\MP{2}$ are $(\Sigma_C\cup\{\N\})$-constrainedly isomorphic;
        \item $\MP{1}$ and $\MP{2}$ are $C$-constrainedly isomorphic.
    \end{compactenum}
\end{corollary}
\begin{proof}
    The equivalence (i)$\leftrightarrow$(ii) is shown in Remark \ref{rem:C-constrained_GH_and_other_GH}(v). The implication (ii)$\to$(iii) follows from Proposition \ref{prop:correspondence_minimizing_distortion_metric_pairs}, and (iv)$\to$(i) is straightforward. Finally, let us consider (iii)$\to$(iv). If $f$ is a $(\Sigma_C\cup\{\N\})$-constrained isomorphism, then the chain of inclusions \eqref{eq:C-constrained_iff_TC-constrained} used in Proposition \ref{prop:C-constrained_iff_TC-constrained} becomes a chain of equalities, and so $f$ is a $C$-constrained isomorphism as well according to Fact \ref{fact:C_n_and_sigma_n}(iv). 
\end{proof}

\begin{remark}\label{rem:isometry_of_finite_point_clouds}
    (i) Corollary \ref{coro:0_distance_chromatic} can be specialized to the setting of metric pairs using the definition provided in Remark \ref{rem:C-constrained_GH_and_other_GH}(ii). It reads that, for every pair of compact metric pairs $(A_1,X_1)$ and $(A_2,X_2)$, there is a pair isometry between them if and only if $\GH((A_1,X_1),(A_2,X_2))=0$.
    
    (ii) Compactness in the statement of Corollary \ref{coro:0_distance_chromatic} is only a sufficient condition to retrieve the bi-implication. Let us consider the following important class of examples. Let $(\R^n,\chi\colon X\to\N)$ and $(\R^n,\psi\colon Y\to\N)$ be two chromatic metric pairs such that $X$ and $Y$ are finite. Assume that $\GH^C((\R^n,\chi),(\R^n,\psi))=0$ where $\N\in C\subseteq\mathcal P(\N)$ is a constraint set closed under intersection. In particular, it is easy to see that $\GH^C(\chi,\psi)=0$. Since $X$ and $Y$ are finite, there is a $C$-constrained isomorphism $f\colon (X,\chi)\to (Y,\psi)$. Since $X$ and $Y$ are finite, the isometry $f$ can be extended to an isometry $\overline f\colon\R^n\to\R^n$ such that $\overline f|_X=f$. Since $f$ is a $C$-constrained isomorphism, so is $\overline f$.

\end{remark}

\section{Stability of the six-packs of persistent diagrams}\label{sec:stability}

We combine the notion of the above defined $C$-constrained Gromov-Hausdorff distance with techniques from \cite{ChoMem3} to strengthen a stability result in persistent homology from \cite{Coh-SteEdeHar, Coh-SteEdeHarMor}, relevant for the study of chromatic point sets using six-packs of persistent diagrams as described in \cite{CulDraEdeSag1}. We recall basic definitions from persistent homology in \S\ref{sub:ph-background}, formulate and prove the stability result in \S\ref{sub:stability}, and finally compare the result to a related stability result from \cite{Mem_tripods} in \S\ref{sub:Cech_stability}.

\subsection{Background on persistent homology}\label{sub:ph-background}

We recall relevant definitions for persistent homology; for a more complete treatment, see, e.g., \cite{EdeHar}. Succinctly, its main idea is to study a filtered topological space by applying a homology functor with field coefficients to the sequence of subspaces with inclusions, and decomposing thus created module into a unique direct sum of indecomposables. One typical example is the study of finite point sets, in particular in a Euclidean space. Consider the sequence of unions of balls centered at the points with radius growing from zero to infinity. Applying homology to this sequence, we study connected components, loops, voids and higher-dimensional analogues of such features in the unions of balls. Each summand in the algebraic decomposition mentioned above gives us an interval of radii within which a certain such feature persists. The collection of all those intervals is a popular invariant in topological data analysis.

For the rest of the section, fix a field $k$. A {\em persistence module}, $\mathcal U=\{U_\delta\xrightarrow{h_{\delta,\nu}}U_\nu\}_{\delta\leq\nu}$, is a family of vector spaces over the field $k$ indexed by $\R$ together with linear maps between them such that: $h_{\delta,\delta}$ is the identity for every $\delta\in\R$, and $h_{\nu,\eta}\circ h_{\delta,\nu}=h_{\delta,\eta}$. Equivalently, it is a functor from $\R$ as a poset category to vector spaces. Often the indexing set $\R$ is replaced by a finite sequence of indices---in that case, we can also think of the persistence module as a representation of a type $A_n$ quiver. More generally, any poset can be the indexing set;  here we always assume indexing by $\R$.

An important example of a persistence module is the \emph{interval module},
$\intmod{[a,b)}=\{I_\delta\xrightarrow{\iota_{\delta,\nu}}I_\nu\}_{\delta\leq\nu}$, for an interval $[a, b)\subseteq \R$,
with $I_\delta$ the one dimensional vector space, $k$, for $\delta\in[a,b)$ and zero otherwise, and $\iota_{\delta,\nu}$ the identity whenever $\delta\leq\nu$ are both in $[a,b)$ and zero otherwise. It is defined the same way for the other types of intervals. We say that a persistence module is \emph{finite dimensional} if each vector space in it is finite dimensional.

Given two persistence modules $\mathcal U=\{U_\delta\xrightarrow{h_{\delta,\nu}}U_\nu\}_{\delta\leq\nu}$ and $\mathcal V=\{V_\delta\xrightarrow{l_{\delta,\nu}}V_\nu\}_{\delta\leq\nu}$, a morphism $\varphi:\mathcal{U}\rightarrow \mathcal{V}$ is a collection of linear maps $\{\varphi_\delta:U_\delta\rightarrow V_\delta\}_{\delta}$ such that all the square diagrams commute: $\varphi_{\nu} h_{\delta,\nu}=l_{\delta,\nu}\varphi_{\delta}$. If we see $\mathcal{U}$ and $\mathcal{V}$ as functors, $\varphi$ is a natural transformation.

Most standard concepts for vector spaces naturally carry over to persistence modules using point-wise definitions. For example, a \emph{direct sum} of persistence modules, $\mathcal{U}\oplus\mathcal{V}$, consists of the direct sums $U_\delta\oplus V_\delta$ and maps $h_{\delta,\nu}\oplus l_{\delta,\nu}$. We can take the same approach for defining isomorphism, injection, surjection or images, kernels and cokernels of morphisms.

Crucially, every finite dimensional persistence module is isomorphic to a direct sum of interval modules, and the collection of the non-trivial interval modules in the sum and their multiplicities are unique (\cite{Gab, Cra-Boe}). This allows us to define a \emph{persistence diagram} of $\mathcal{U}$, denoted by $\Dgm(\mathcal{U})$, as the multiset containing the pair of endpoints, $(a, b)$, for each interval module in the decomposition given by an interval with those endpoints\footnote{
    For various reasons one might or might not want to include the diagonal elements $(a,a)$. We assume that no diagonal elements are part of the diagrams. The common alternatives are to consider all diagonal elements with infinite multiplicities or a single abstract element representing the whole diagonal with infinite multiplicity to be parts of the diagram; or to leave complete freedom of whether diagonal elements are or are not in a diagram. The particular choice is inconsequential, and the distance between diagrams differing only in this choice is zero for all commonly defined distances.
}. Generally, this is an invariant losing the information about the type of the intervals. Note, however, that all the persistence diagrams considered later have all the intervals of the type $[a. b)$.

To discuss stability results, we recall a standard notion of distance for persistence modules and persistence diagrams. Given two persistence modules, $\mathcal{U}$ and $\mathcal{V}$, as above, an {\em $\varepsilon$-interleaving} between them consists of two families of linear maps 
\[
f_{\delta}\colon U_\delta\to U_{\delta+\varepsilon}, \quad \text{and} \quad g_{\delta}\colon V_\delta\to V_{\delta+\varepsilon}
\]
for each $\delta\in\R$, such that the following four diagrams commute for every pair of indices $\delta\leq\nu$:

\begin{tabular}{cc}
    \begin{tikzcd}[row sep=normal, column sep=normal]
        U_\delta \arrow[rr, "h_{\delta,\nu}"] \arrow[dr, "f_\delta"]
        & & U_\nu \arrow[dr,"f_\nu"] &\\
        & V_{\delta+\varepsilon}  \arrow[rr, "l_{\delta+\varepsilon,\nu+\varepsilon}"] & & V_{\nu+\varepsilon}
    \end{tikzcd}
    &
    \begin{tikzcd}[row sep=normal, column sep=normal]
        & & U_{\delta+\varepsilon} \arrow[rr, "h_{\delta+\varepsilon,\nu+\varepsilon}"] 
        & & U_{\nu+\varepsilon}  &\\
        & V_\delta  \arrow[rr, "l_{\delta,\nu}"] \arrow[ur, "g_\delta"] & & V_\nu \arrow[ur, "g_\nu"]
    \end{tikzcd}
    \vspace{5mm}\\
    \begin{tikzcd}[row sep=normal, column sep=normal]
        U_\delta \arrow[rr, "h_{\delta,\delta+2\varepsilon}"] \arrow[dr, "f_\delta"]
        & & U_{\delta+2\epsilon}  &\\
        & V_{\delta+\varepsilon}  \arrow[ur, "g_{\delta+\varepsilon}"]
    \end{tikzcd}
    &
    \begin{tikzcd}[row sep=normal, column sep=normal]
        & & U_{\delta+\varepsilon} \arrow[dr, "f_{\delta+\varepsilon}"]  &\\
        & V_\delta  \arrow[ur, "g_\delta"] \arrow[rr, "l_{\delta,\delta+2\varepsilon}"]  & & V_{\delta+2\varepsilon}
    \end{tikzcd}
\end{tabular}

\noindent The interleaving distance between the persistence modules is then defined as \[
    \dint(\mathcal{U}, \mathcal{V}) = \inf\setdef{\epsilon\geq 0}{\text{an $\varepsilon$-interleaving between $\mathcal{U}$ and $\mathcal{V}$ exists}}.
\]
Note, in particular, that the interleaving distance between $\intmod{[a,b)}$ and $\intmod{[a-\varepsilon,b+\varepsilon)}$ is $\varepsilon$.

Next, let us recall that, given two multisets $Z$ and $W$ of elements in $\R^2$, their {\em bottleneck distance} is defined as
\[
    d_B(Z,W)=\inf_{\substack{Z'\subseteq Z\\W'\subseteq W}}\inf_{\substack{f\colon Z'\to W'\\\text{bijection}}}\max\{\sup_{z\in Z'}\lvert\lvert z-f(z)\rvert\rvert,\sup_{z\in Z\setminus Z'}\dist(z,\Delta),\sup_{w\in W\setminus W'}\dist(w,\Delta)\},
\]
where $\Delta=\{(x,x)\mid x\in\R^2\}$ denotes the diagonal of $\R^2$, and $\lvert\lvert-\rvert\rvert$ and $\dist(-,-)$
are $\ell_\infty$ norm and distance. The bottleneck distance on persistence diagrams and the interleaving distance on persistence modules are related:

\begin{theorem}[Algebraic stability theorem, \cite{ChaCoh-SteGliGuiOud}]\label{thm:algebraic_stability_theorem}
    For finitely dimensional persistence modules $\mathcal U$ and $\mathcal V$,
    \[
        d_B(\Dgm(\mathcal U),\Dgm(\mathcal V))\leq\dint(\mathcal{U}, \mathcal{V}).
    \]
\end{theorem}

\subsection{Stability result}\label{sub:stability}

In this section, we give definitions of \v{C}ech complexes and their $\Gamma$-subcomplexes, recall the definition of a six-pack of persistent diagrams, and show their stability with respect to suitable constrained Gromov-Hausdorff distances.

\paragraph{\v{Cech} complexes and filtrations.} For a metric pair $(A,X)$, we denote by $\check{C}(X;A)$ the {\em ambient \v{C}ech filtration} of $(A,X)$, i.e., the filtration $\{\check{C}_\delta(X;A)\xrightarrow{i_{\delta,\varepsilon}}\check{C}_\varepsilon(X;A)\}_{\delta\leq\varepsilon}$,
where the objects $\check{C}_\nu(X;A)$, for $\nu\geq 0$, are simplicial complexes defined as the collection of simplices $\sigma\subseteq X$ for which there exists $a\in A$ such that $d(a,x)\leq \nu$ for all $x\in\sigma$, and the maps $i_{\delta,\varepsilon}$ are inclusions.

Let $\extMP{}$ be a chromatic metric pair, and $\Gamma$ be a simplicial complex with vertices in $\N$. 
Based on \cite[Section~3.4]{CulDraEdeSag1}, we define the \emph{$\Gamma$-subcomplex} of the \v{C}ech complex at radius $\delta$, $\check{C}_\delta(X;A)$, as 
\[
\check{C}_\delta^\Gamma(\chi; A) = \left\{ \mu\in\check{C}_\delta(X;A) \ \middle|\ \chi(\mu)\in\Gamma \right\},
\]
and we denote by $\check{C}^\Gamma(\chi; A)$ the filtration $\left\{ \check{C}_\delta^\Gamma(\chi; A) \right\}_{\delta\geq0}$. 
Note that the coloring map $\chi$ naturally extends to a simplicial map from any complex with vertices in $X$ to a complete complex on $\N$ (or on $\im\chi$). Keeping the symbol $\chi$ to also denote this extension, we can write $\check{C}_\delta^\Gamma(\chi; A) = \chi^{-1}(\Gamma) \cap \check{C}_\delta(X;A)$.
Up to a natural homotopy equivalence, the complexes $\check{C}_\delta^\Gamma(\chi; A)$ are the same as the $\Gamma$-subcomplexes of chromatic alpha complexes studied in \cite{CulDraEdeSag1}. Indeed, in the proof of \cite[Theorem~3.8]{CulDraEdeSag1}, we can replace the Voronoi planks defined there by their uncropped versions, and replace the right side of the diagrams with the \v{C}ech versions of the complexes, while keeping the left side of the diagrams the same.

\paragraph{Six-packs of persistent diagrams.} Let $\extMP{}$ be a chromatic metric pair, and $\Lambda\leq\Gamma$ be two simplicial complexes with vertices in $\N$. Additionally, choose a degree $p\in\N$. We give definitions of six persistence modules parametrized by this fixed data. Consider the two filtrations and inclusion between them $\iota: \check{C}^\Lambda(\chi; A) \xrightarrow{} \check{C}^\Gamma(\chi; A)$, defined pointwise as $\iota_\delta: \check{C}_\delta^\Lambda(\chi; A) \xrightarrow{\subseteq} \check{C}_\delta^\Gamma(\chi; A)$. Denote by $H_p(\check{C}^\Gamma(\chi; A))$ the persistence module obtained by applying the degree $p$ homology functor, $H_p$, to the filtration $\check{C}^\Gamma(\chi; A)$, and analogously for $\Lambda$. Furthermore, apply $H_p$ to $\iota$ and denote by $\hat\iota=H_p(\iota):H_p(\check{C}^\Lambda(\chi; A))\rightarrow H_p(\check{C}^\Gamma(\chi; A))$ the morphism between the two persistence modules. Recall that we can define image, kernel and cokernel of $\hat\iota$ pointwise. That is, e.g., $\ker\hat\iota$ is a persistence module with vector spaces $(\ker\hat\iota)_\delta = \ker\hat\iota_\delta$ and maps the restrictions of the maps induced by the original filtration inclusions. Finally, consider the two \v{C}ech filtrations as a pair and apply the relative homology to get a persistence module $H_p(\check{C}^\Gamma(\chi; A), \check{C}^\Lambda(\chi; A))$. Altogether, we have six persistence modules:

\smallskip
\begin{compactenum}[-]
    \item $\Dom{A}{\chi}{\Lambda}{\Gamma}{p} = \dom \hat\iota = H_p(\check{C}^\Lambda(\chi; A))$,

    \item $\Cod{A}{\chi}{\Lambda}{\Gamma}{p} = \codom \hat\iota = H_p(\check{C}^\Gamma(\chi; A))$,

    \item $\Img{A}{\chi}{\Lambda}{\Gamma}{p} = \im \hat\iota$,

    \item $\Ker{A}{\chi}{\Lambda}{\Gamma}{p} = \ker \hat\iota$,

    \item $\Cok{A}{\chi}{\Lambda}{\Gamma}{p} = \coker \hat\iota$,

    \item $\Rel{A}{\chi}{\Lambda}{\Gamma}{p} = H_p(\check{C}^\Gamma(\chi; A), \check{C}^\Lambda(\chi; A))$.
\end{compactenum}
The degree $p$ six-pack of persistence diagrams is then the collection of the corresponding persistence diagrams:
\begin{multline*}
    \{
        \Dgm \Dom{A}{\chi}{\Lambda}{\Gamma}{p},\ 
        \Dgm \Cod{A}{\chi}{\Lambda}{\Gamma}{p},\ 
        \Dgm \Img{A}{\chi}{\Lambda}{\Gamma}{p},\  \\
        \Dgm \Ker{A}{\chi}{\Lambda}{\Gamma}{p},\ 
        \Dgm \Cok{A}{\chi}{\Lambda}{\Gamma}{p},\ 
        \Dgm \Rel{A}{\chi}{\Lambda}{\Gamma}{p}
    \}
\end{multline*}

\paragraph{The stability.} For a finite-dimensional simplicial complex $\Gamma$ with vertex set in $\N$, we define $C(\Gamma)$ as the set of maximal faces of $\Gamma$. This will be the right strength of constraints for the stability of six-packs of persistence diagrams.

\begin{lemma}
	Let $\Gamma$ be a finite-dimensional simplicial complex with vertices in $\N$. Then, every $C(\Gamma)$-constrained map $f\colon\extMP{1}\to\extMP{2}$ between chromatic metric pairs induces a simplicial map
	\[
		f_\delta^\Gamma\colon\check{C}_\delta^\Gamma(\chi_1;A_1)\to\check{C}_{\delta+\dis(f)}^\Gamma(\chi_2;A_2)
	\]
	for every $\delta\geq 0$.
\end{lemma}
\begin{proof}
	For every $\delta\geq 0$, we define $f_\delta\colon\check{C}_\delta(X_1;A_1)\to\check{C}_{\delta+\dis(f)}(X_2;A_2)$ on the vertex set by setting $f_\delta(x)=f(x)$ for every vertex $x$ of $\Gamma$. 
    First, we prove that it is well-defined, and then that its restriction $f_\delta^\Gamma=f_\delta|_{\check{C}_\delta^\Gamma(\chi_1;A_1)}$ maps into $\check{C}^\Gamma_{\delta+\dis(f)}(\chi_2;A_2)$, i.e., for every simplex $\mu$, $\chi_1(\mu)\in\Gamma$ implies $\chi_2(f(\mu))\in\Gamma$.
	
	To prove that $f_\delta$ is well-defined, pick $\sigma\in\check{C}_\delta(X_1;A_1)$ and $a\in A_1$ such that $d_1(x,a)\leq\delta$ for every $x\in\sigma$. Then,
	\[
		d_2(f(x),f(a))\leq d_1(x,a)+\dis f\leq\delta+\dis f,
	\]
	and so $f(\sigma)\in\check{C}_{\delta+\dis(f)}(X_2;A_2)$.
	
	Let now $\mu\in\check{C}_\delta(X_1;A_1)$ such that $\chi_1(\mu)\in\Gamma$. 
    Then, $\chi_1(\mu)\subseteq \tau\in C(\Gamma)$ for some maximal simplex~$\tau$. 
    
    Hence, $\mu\subseteq \chi_1^{-1}(\tau)$, and so $f(\mu)\subseteq \chi_2^{-1}(\tau)$. Rewriting the latter containment, we obtain that $\chi_2(f_\delta(\mu))=\chi_2(f(\mu))\subseteq\tau$, and so $\chi_2(f_\delta(\mu))\in\Gamma$.
\end{proof}

Let us recall that two maps $f,g\colon K\to L$ between simplicial complexes are {\em contiguous} if $f(\sigma)\cup g(\sigma)\in L$ for every $\sigma\in K$.
\begin{lemma} \label{lem:diagrams_gamma}
	Let $\MP{1}$ and $\MP{2}$ be two chromatic metric pairs and $\Gamma$ be a finite-dimensional simplicial complex with vertices in $\N$. Let $f^1\colon\MP{1}\to\MP{2}$ and $f^2\colon\MP{2}\to\MP{1}$ be two $C(\Gamma)$-constrained maps such that $\max\{\dis f^1,\dis f^2,\codis(f^1,f^2)\}\leq2\varepsilon$. 
    Then, the following diagrams, for $0\leq\delta<\nu$,

    \begin{center}
    \begin{tikzcd}[row sep=small, column sep=-3mm]
    \check{C}^\Gamma_\delta(\chi_1;A_1) \arrow[rr, "i_{\delta,\nu}"] \arrow[dr, "f^1_\delta"]
    & & \check{C}^\Gamma_{\nu}(\chi_1;A_1) \arrow[dr,"f^1_\nu"] &\\
    & \check{C}^\Gamma_{\delta+2\varepsilon}(\chi_2;A_2)  \arrow[rr, "i_{\delta+2\varepsilon,\nu+2\varepsilon}"] & & \check{C}^\Gamma_{\nu+2\varepsilon}(\chi_2;A_2)
    \end{tikzcd}\hfill
    \begin{tikzcd}[row sep=small, column sep=-3mm]
    & & \check{C}^\Gamma_{\delta+2\varepsilon}(\chi_1;A_1) \arrow[rr, "i_{\delta+2\varepsilon,\nu+2\varepsilon}"] 
    & & \check{C}^\Gamma_{\nu+2\varepsilon}(\chi_1;A_1)  &\\
    & \check{C}^\Gamma_{\delta}(\chi_2;A_2)  \arrow[rr, "i_{\delta,\nu}"] \arrow[ur, "f^2_\delta"] & & \check{C}^\Gamma_\nu(\chi_2;A_2) \arrow[ur, "f^2_\nu"]
    \end{tikzcd}
    \end{center}
    
    \noindent commute, and the following ones

    \begin{center}
    \begin{tikzcd}[row sep=small, column sep=-2mm]
    \check{C}^\Gamma_\delta(\chi_1;A_1) \arrow[rr, "i_{\delta,\delta+4\varepsilon}"] \arrow[dr, "f^1_\delta"]
    & & \check{C}^\Gamma_{\delta+4\varepsilon}(\chi_1;A_1)  &\\
    & \check{C}^\Gamma_{\delta+2\varepsilon}(\chi_2;A_2)  \arrow[ur, "f^2_{\delta+2\varepsilon}"]  \\
    \end{tikzcd}\hspace{5mm}
    \begin{tikzcd}[row sep=small, column sep=-2mm]
    & & \check{C}^\Gamma_{\delta+2\varepsilon}(\chi_1;A_1) \arrow[dr, "f^1_{\delta+2\varepsilon}"]  &\\
    & \check{C}^\Gamma_\delta(\chi_2;A_2)  \arrow[ur, "f^2_\delta"] \arrow[rr, "i_{\delta,\delta+4\varepsilon}"]  & & \check{C}^\Gamma_{\delta+4\varepsilon}(\chi_2;A_2)  \\
    \end{tikzcd}
    \end{center}
    
    \noindent commute up to contiguity.
\end{lemma}

\begin{proof}
	The proof is a straightforward adaptation of that of \cite[Proposition 15]{ChoMem3}.
\end{proof}

Let $(A,\chi)$ be a chromatic metric pair. If $\Gamma$ and $\Lambda$ are two finite-dimensional simplicial complexes with vertices in $\N$ such that $\Lambda\subseteq\Gamma$, then there is a canonical inclusion $j_\delta\colon\check{C}^\Lambda_\delta(\chi;A)\to\check{C}_\delta^\Gamma(\chi;A)$ for every $\delta\geq 0$. The following result is an easy consequence of what we have constructed so far.

\begin{lemma} \label{lem:diagrams_gamma-lambda}
    Let $\MP{1}$ and $\MP{2}$ be two chromatic metric pairs and $\Lambda\subseteq \Gamma$ be two finite-dimensional simplicial complexes with vertices in $\N$. Let $f^1\colon\MP{1}\to\MP{2}$ and $f^2\colon\MP{2}\to\MP{1}$ be two $C(\Gamma)\cup C(\Lambda)$-constrained maps and $\max\{\dis f^1,\dis f^2,\codis(f^1,f^2)\}\leq 2\varepsilon$. Then, all the squares and triangles in the following diagrams (where we removed the indices where there is no risk of ambiguity, for the sake of simplicity) commute up to contiguity:

    \begin{center}
    \begin{tikzcd}[row sep=small, column sep=small]
    \check{C}^\Gamma_\delta(\chi_1;A_1) \arrow[rr, "i"] \arrow[dr, "f^1"]
    & & \check{C}^\Gamma_{\nu}(\chi_1;A_1) \arrow[dr,"f^1"]
    \arrow[from=dd, near start, "j"] &\\
    & \check{C}^\Gamma_{\delta+2\varepsilon}(\chi_2;A_2)  \arrow[rr, near start, crossing over, "i"] & & \check{C}^\Gamma_{\nu+2\varepsilon}(\chi_2;A_2)  \\
    \check{C}^\Lambda_\delta(\chi_1;A_1) \arrow[rr, near end, "i"] \arrow[dr, "f^1"] \arrow[uu, "j"] & & \check{C}^\Lambda_\nu(\chi_1;A_1) \arrow[dr, "f^1"]  & \\
    & \check{C}^\Lambda_{\delta+2\varepsilon}(\chi_2;A_2) \arrow[rr, "i"] \arrow[uu, crossing over, near end, "j"] & & \check{C}^\Lambda_{\nu+2\varepsilon}(\chi_2;A_2) \arrow[uu, "j"] \\
    \end{tikzcd}

    \begin{tikzcd}[row sep=small, column sep=small]
    & & \check{C}^\Gamma_{\delta+2\varepsilon}(\chi_1;A_1) \arrow[from=dd, near start, "j"] \arrow[rr, "i"] 
    & & \check{C}^\Gamma_{\nu+2\varepsilon}(\chi_1;A_1)  &\\
    & \check{C}^\Gamma_{\delta}(\chi_2;A_2)  \arrow[rr, crossing over, near end, "i"] \arrow[ur, "f^2"] & & \check{C}^\Gamma_\nu(\chi_2;A_2) \arrow[ur, "f^2"] \\
    & & \check{C}^\Lambda_{\delta+2\varepsilon}(\chi_1;A_1) \arrow[rr, near start, "i"]  & & \check{C}^\Lambda_{\nu+2\varepsilon}(\chi_1;A_1) \arrow[uu, "j"]  & \\
    & \check{C}^\Lambda_\delta(\chi_2;A_2) \arrow[ur, "f^2"] \arrow[rr, "i"] \arrow[uu, "j"]  & & \check{C}^\Lambda_\nu(\chi_2;A_2) \arrow[uu, crossing over, near start, "j"] \arrow[ur, "f^2"]  \\
    \end{tikzcd}
    \end{center}

    \begin{tikzcd}[row sep=small, column sep=small]
    \check{C}^\Gamma_\delta(\chi_1;A_1) \arrow[rr, "i"] \arrow[dr, "f^1"]
    & & \check{C}^\Gamma_{\delta+4\varepsilon}(\chi_1;A_1)  &\\
    & \check{C}^\Gamma_{\delta+2\varepsilon}(\chi_2;A_2)  \arrow[ur, "f^2"]  \\
    \check{C}^\Lambda_\delta(\chi_1;A_1) \arrow[rr, near end, "i"] \arrow[dr, "f^1"] \arrow[uu, "j"] & & \check{C}^\Lambda_{\delta+4\varepsilon}(\chi_1;A_1) \arrow[uu, "j"] & \\
    & \check{C}^\Lambda_{\delta+2\varepsilon}(\chi_2;A_2) \arrow[uu, crossing over, near end, "j"] \arrow[ur, "f^2"] \\
    \end{tikzcd}
    \hspace{-1cm}
    \begin{tikzcd}[row sep=small, column sep=small]
    & & \check{C}^\Gamma_{\delta+2\varepsilon}(\chi_1;A_1) \arrow[dr, "f^1"] \arrow[from=dd, near start, "j"]  &\\
    & \check{C}^\Gamma_\delta(\chi_2;A_2)  \arrow[ur, "f^2"] \arrow[rr, crossing over, near start, "i"]  & & \check{C}^\Gamma_{\delta+4\varepsilon}(\chi_2;A_2)  \\
    & & \check{C}^\Lambda_{\delta+2\varepsilon}(\chi_1;A_1) \arrow[dr, "f^1"] & \\
    & \check{C}^\Lambda_\delta(\chi_2;A_2) \arrow[uu, crossing over, near end, "j"] \arrow[ur, "f^2"] \arrow[rr, "i"] & & \check{C}^\Lambda_{\delta+4\varepsilon}(\chi_2;A_2) \arrow[uu, "j"] \\
    \end{tikzcd}
\end{lemma}

\begin{theorem}\label{thm:stability_6_pack}
	Let $\extMP{1}$ and $\extMP{2}$ be two chromatic metric pairs with $X_1$ and $X_2$ finite, $\Lambda\leq \Gamma$ be finite-dimensional simplicial complexes on $\N$, and $C=C(\Gamma)\cup C(\Lambda)$. Choose $p\in\N$. Let $\mathcal{U}_1$ be one of the persistence modules
    $\Dom{A_1}{\chi_1}{\Lambda}{\Gamma}{p}$,
    $\Cod{A_1}{\chi_1}{\Lambda}{\Gamma}{p}$,
    $\Img{A_1}{\chi_1}{\Lambda}{\Gamma}{p}$,
    $\Ker{A_1}{\chi_1}{\Lambda}{\Gamma}{p}$,
    $\Cok{A_1}{\chi_1}{\Lambda}{\Gamma}{p}$,
    $\Rel{A_1}{\chi_1}{\Lambda}{\Gamma}{p}$,
    and $\mathcal{U}_2$ be the corresponding persistence module for the pair $(A_2, \chi_2)$. Then
	\[
	d_B(\Dgm\,\mathcal{U}_1,\Dgm\,\mathcal{U}_2)\leq 2\GH^{C}(\MP{1},\MP{2}).
	\]

\end{theorem}

\begin{proof}
    Let $d_{GH}^C(\MP{1},\MP{2})<\varepsilon$. Let $f^1\colon\MP{1}\to\MP{2}$ and $f^2\colon\MP{2}\to\MP{1}$ be two $C$-constrained maps such that $\max\{\dis f^1,\dis f^2,\codis(f^1,f^2)\}\leq 2\varepsilon$.
    We argue that all the persistence modules of the six-pack are $2\varepsilon$-interleaved, from which the result follows according to Theorem \ref{thm:algebraic_stability_theorem}. 

    By \cite[Theorem 12.5]{Mun}, applying the homology functor $H_p$ to the diagrams described in Lemma \ref{lem:diagrams_gamma} yields commutative diagrams. The lower and upper layers are the desired $2\varepsilon$-interleavings for $\Dom{A_i}{\chi_i}{\Lambda}{\Gamma}{p}$ and $\Cod{A_i}{\chi_i}{\Lambda}{\Gamma}{p}$, respectively.

    Consider the kernels of the upwards inclusions, $j$, for each $\delta$. By standard diagram chasing, the restrictions of the maps $f^1$, $f^2$ and $i$ define maps between the kernels. These restrictions form the $2\varepsilon$-interleaving of $\Ker{A_i}{\chi_i}{\Lambda}{\Gamma}{p}$. Analogously, $2\varepsilon$-interleavings for $\Img{A_i}{\chi_i}{\Lambda}{\Gamma}{p}$ and $\Cok{A_i}{\chi_i}{\Lambda}{\Gamma}{p}$ are constructed.

    As for the relative case, note that the upper-layer of the aforementioned diagrams commutes up to {\em contiguity of pairs} with respect to the pairs $(\check{C}^\Gamma(\chi_i;A_i), \check{C}^\Lambda(\chi_i;A_i))$ for $i=1,2$ as defined in \cite{Mun}. By \cite[Theorem 12.6]{Mun}, the diagrams for the corresponding relative homology groups commute yielding the desired $2\varepsilon$-interleaving for $\Rel{A_i}{\chi_i}{\Lambda}{\Gamma}{p}$.
\end{proof}

Note that together with Corollary~\ref{coro:GH_metric_pairs}, the theorem is a direct generalization of a previous result \cite[Proposition~2.4]{TorGon} showing the same stability for degree 0 image, kernel and cokernel persistence of metric pairs; they use Vietoris-Rips complexes, which coincide with \v{Cech} complexes for degree 0.

In Example \ref{ex:chromatic_invariants} we showed how the stable invariants defined in Definition \ref{def:cchromatic_invariants} can be used to lower bound the $C$-constrained Gromov-Hausdorff distance. In Example \ref{ex:lower_bound_from_stability_6pack}, we describe a similar application for Theorem \ref{thm:stability_6_pack}. Exploiting the stability of Vietoris-Rips persistence diagrams to bound the usual Gromov-Hausdorff distance has already been done in different contexts, e.g., in \cite{LimMemOku} while estimating the Gromov-Hausdorff distance between spheres of different dimensions, and in \cite{PatMonGarGorNeu} to discuss isometries between languages.

\begin{example}        \label{ex:lower_bound_from_stability_6pack}
    Given two chromatic metric pairs, $\MP{1}$ and $\MP{2}$, and a constraint set $\N\in C\subseteq\mathcal P(\N)$, we can bound their $C$-constrained Gromov-Hausdorff distance using Theorem \ref{thm:stability_6_pack} as follows. Consider two subsets $G$ and $L$ of $\Tau_C$ consisting of finite subsets of $\N$ and such that $L$ is contained in $\downarrow G=\{\sigma\subseteq\N\mid\exists\tau\in G:\sigma\subseteq\tau\}$. Define $\Gamma=\downarrow G$ and $\Lambda=\downarrow L$, which are finite-dimensional simplicial complexes by construction and satisfy $\Lambda\leq\Gamma$. Note that $C(\Lambda)\subseteq L$ and $C(\Gamma)\subseteq G$, and so $C(\Gamma)\cup C(\Lambda)\subseteq\Tau_C$. Therefore,
    \[
    \begin{aligned}
    \GH^C(\MP{1},\MP{2})&\,=\GH^{\Tau_C}(\MP{1},\MP{2})\geq \GH^{C(\Gamma)\cup C(\Lambda)}(\MP{1},\MP{2})\geq\\
    &\,\geq\frac{1}{2}d_B(\Dgm_1,\Dgm_2),
    \end{aligned}
    \]
    where $\Dgm_1$ is one of the diagrams from the degree $p$ six-pack of persistence diagrams of the filtered pair $(\check{C}^\Gamma(\chi_1;A_1),\check C^\Lambda(\chi_1;A_1))$ and $\Dgm_2$ is the corresponding diagram for $(\check{C}^\Gamma(\chi_2;A_2),\check C^\Lambda(\chi_2;A_2))$.
    
    We provide two chromatic metric spaces $\chi_1$ and $\chi_2$ and a constraint set $\N\in C\subseteq\mathcal P(\N)$ such that the previous approach provides the optimal lower bound even if $\chi_1^{-1}(\sigma)$ and $\chi_2^{-1}(\sigma)$ are isometric for every $\sigma\in C$.

    Fix $r>0$. Define $X_1=X_2=[0,2r]\cup\{3r\}$,
    \[
    \chi_1\colon\begin{cases}\begin{aligned}
        [0,r)\cup\{2r\} &\mapsto 0\\
        [r,2r)\cup\{3r\} &\mapsto 1,
    \end{aligned}\end{cases}\quad\quad \chi_2\colon\begin{cases}\begin{aligned}
        [r,2r)\cup\{3r\} &\mapsto 0,\\
        [0,r)\cup\{2r\} &\mapsto 1,
        \end{aligned}\end{cases}
    \]
    (see Figure \ref{fig:lower_bound_from_stability_6pack_dim0}), and $C=\{\{0\},\{0,1\}\}$.  Clearly, $\chi_1^{-1}(\{0\})$ is isometric to $\chi_2^{-1}(\{0\})$. We define $L=\{\{0\}\}$ and $G=\{\{0,1\}\}$, and so $\Lambda=\{\{0\}\}$ and $\Gamma=\{\{0\},\{1\},\{0,1\}\}$.
        \begin{figure}[h!]
        \centering
        \begin{tikzpicture}
            \draw[lightgray] (0,0)--(4.5,0);
            \draw (0,0.1) node[above]{$0$};
            \draw (1.5,0.1) node[above]{$r$};
            \draw (3,0.1) node[above]{$2r$};
            \draw (4.5,0.1) node[above]{$3r$};
            \draw[cyan, ultra thick] (1.5,0)--(3,0);
            \fill[cyan] (4.5,0) circle (2pt);
            \draw[magenta, ultra thick] (0,0)--(1.5,0);
            \fill[magenta] (3,0) circle (2pt);
            \draw (2.25,-0.4) node[below]{$\chi_1$};
            \fill[magenta] (0,0) circle (2pt);
            \fill[cyan] (1.5,0) circle (2pt);
            
            \draw[lightgray] (6,0)--(10.5,0);
            \draw (6,0.1) node[above]{$0$};
            \draw (7.5,0.1) node[above]{$r$};
            \draw (9,0.1) node[above]{$2r$};
            \draw (10.5,0.1) node[above]{$3r$};
            \draw[magenta, ultra thick] (7.5,0)--(9,0);
            \fill[magenta] (10.5,0) circle (2pt);
            \draw[cyan, ultra thick] (6,0)--(7.5,0);
            \fill[cyan] (9,0) circle (2pt);
            \draw (8.25,-0.4) node[below]{$\chi_2$};
            \fill[cyan] (6,0) circle (2pt);
            \fill[magenta] (7.5,0) circle (2pt);
        \end{tikzpicture}
        \caption{\footnotesize A representation of the two chromatic metric pairs described in Example \ref{ex:lower_bound_from_stability_6pack}. The colors $0$ and $1$ are represented in magenta and cyan, respectively.}
        \label{fig:lower_bound_from_stability_6pack_dim0}
    \end{figure}
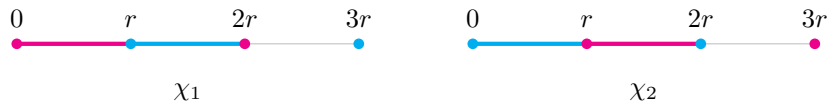
    Let $\mathbb X_i$ be the filtered pair $(\check{C}^\Gamma(\chi_i;X_i),\check{C}^\Lambda(\chi_i;X_i))$ for $i=1,2$\footnote{
    More precisely, to compute the six-pack, we consider an arbitrarily dense uniform discretization of $[0,2r)\cup\{3r\}$.}
    . Then, $\Dgm\Ker{A_1}{\chi_1}{\Lambda}{\Gamma}{0}$ has a single feature $(0,r)$, whereas $\Dgm\Ker{A_2}{\chi_2}{\Lambda}{\Gamma}{0}$ has none. Therefore, $\GH^C(\MP{1},\MP{2})\geq r/2$.

    We claim that $\GH^C(\MP{1},\MP{2})=r/2$. Let us construct two $C$-constrained maps $f\colon\chi_1\to\chi_2$ and $g\colon\chi_2\to\chi_1$. We set, for every $x\in[0,2r]\cup\{3r\}$,
    \[
    f(x)=\begin{cases}
        \begin{aligned}
            &r+x &\text{if $\chi_1(x)=0$,}\\
            &x &\text{otherwise,}
        \end{aligned}
    \end{cases}\quad\text{ and }\quad 
    g(x)=\begin{cases}
        \begin{aligned}
            & x-r &\text{if $\chi_2(x)=0$,}\\
            & x &\text{otherwise}
        \end{aligned}
    \end{cases}
    \]
    (see Figure \ref{fig:lower_bound_from_stability_6pack_maps} for their representation). It is easy to see that $\max\{\dis f,\dis g,\codis(f,g)\}=r$, which implies the claimed upper bound.

    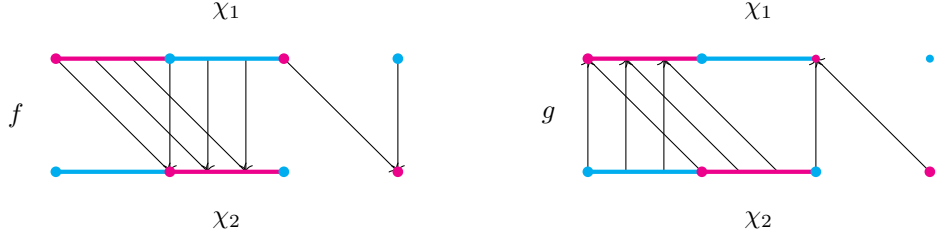
\begin{figure}[h!]
        \centering
        \begin{tikzpicture}
            \draw[->] (0,0)--(1.5,-1.5);
            \draw[->] (0.5,0)--(2,-1.5);
            \draw[->] (1,0)--(2.5,-1.5);
            \draw[->] (1.5,0)--(1.5,-1.5);
            \draw[->] (2,0)--(2,-1.5);
            \draw[->] (2.5,0)--(2.5,-1.5);
            \draw[->] (3,0)--(4.5,-1.5);
            \draw[->] (4.5,0)--(4.5,-1.5);

            \draw (-0.3,-0.75) node[left]{$f$};
            
            \draw[cyan, ultra thick] (1.5,0)--(3,0);
            \fill[cyan] (4.5,0) circle (2pt);
            \draw[magenta, ultra thick] (0,0)--(1.5,0);
            \fill[magenta] (3,0) circle (2pt);
            \draw (2.25,0.4) node[above]{$\chi_1$};
            \fill[magenta] (0,0) circle (2pt);
            \fill[cyan] (1.5,0) circle (2pt);
            
            \draw[magenta, ultra thick] (1.5,-1.5)--(3,-1.5);
            \fill[magenta] (4.5,-1.5) circle (2pt);
            \draw[cyan, ultra thick] (0,-1.5)--(1.5,-1.5);
            \fill[cyan] (3,-1.5) circle (2pt);
            \draw (2.25,-1.9) node[below]{$\chi_2$};
            \fill[cyan] (0,-1.5) circle (2pt);
            \fill[magenta] (1.5,-1.5) circle (2pt);

            \draw[->] (7,-1.5)--(7,0);
            \draw[->] (7.5,-1.5)--(7.5,0);
            \draw[->] (8,-1.5)--(8,0);
            \draw[->] (8.5,-1.5)--(7,0);
            \draw[->] (9,-1.5)--(7.5,0);
            \draw[->] (9.5,-1.5)--(8,0);
            \draw[->] (10,-1.5)--(10,0);
            \draw[->] (11.5,-1.5)--(10,0);            
            \draw (6.7,-0.75) node[left]{$g$};

            \draw[cyan, ultra thick] (8.5,0)--(10,0);
            \fill[cyan] (11.5,0) circle (1.5pt);
            \draw[magenta, ultra thick] (7,0)--(8.5,0);
            \fill[magenta] (10,0) circle (1.5pt);
            \draw (9.25,0.4) node[above]{$\chi_1$};
            \fill[magenta] (7,0) circle (2pt);
            \fill[cyan] (8.5,0) circle (2pt);
            
            \draw[magenta, ultra thick] (8.5,-1.5)--(10,-1.5);
            \fill[magenta] (11.5,-1.5) circle (2pt);
            \draw[cyan, ultra thick] (7,-1.5)--(8.5,-1.5);
            \fill[cyan] (10,-1.5) circle (2pt);
            \draw (9.25,-1.9) node[below]{$\chi_2$};
            \fill[cyan] (7,-1.5) circle (2pt);
            \fill[magenta] (8.5,-1.5) circle (2pt);
        \end{tikzpicture}
        \caption{\footnotesize A representation of the $C$-cconstrained maps $f$ (on the left-hand side) and $g$ (on the right-hand side) constructed in Example \ref{ex:lower_bound_from_stability_6pack} to provide an upper bound to the $C$-constrained Gromov-Hausdorff distance between $\chi_1$ and $\chi_2$.}
        \label{fig:lower_bound_from_stability_6pack_maps}
    \end{figure}
\end{example}

\subsection{Stability of the ambient \v{C}ech persistence diagrams}\label{sub:Cech_stability}

As a particular case of Theorem \ref{thm:stability_6_pack}, we obtain the following result.
\begin{corollary}\label{coro:stability_ambient_cech}
	Let $(A_1,X_1)$ and $(A_2,X_2)$ be two metric pairs such that $X_1$ and $X_2$ are finite. Then, for every dimension $k$,
	\[
	d_B(\Dgm_k(\check{C}(X_1;A_1)),\Dgm_k(\check{C}(X_2;A_2)))\leq 2\GH((A_1,X_1),(A_2,X_2)).
	\]
\end{corollary}

There is an alternative strategy to obtain the same result under some mild technical assumptions. Following \cite{Mem_tripods}, a {\em filtered space} consists of a pair $(X,F_X)$ of a finite set $X$ and a function $F_X\colon\mathcal P(X)\to\R$, called {\em filtration}, that is monotone with respect to inclusions (i.e., $F_X(A)\leq F_X(B)$ provided that $A\subseteq B\subseteq X$).
In \cite{Mem_tripods}, the author defines a distance $d_{\mathcal F}$ between two such filtered spaces $(X,F_X)$ and $(Y,F_Y)$. 
A {\em tripod} $(Z,f,g)$ between $X$ and $Y$ consists of a finite set $Z$ and two surjective maps
    
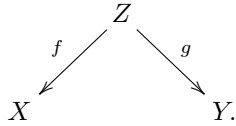
\[
    \xymatrix{
        & Z\ar^{g}[dr]\ar_{f}[dl] &\\
        X & & Y.
    }
    \]
Then, 
\[
d_{\mathcal F}((X,F_X),(Y,F_Y)
)=\newinf_{(Z,f,g)\text{ tripod}}\sup_{A\in\mathcal P(Z)}\lvert F_X(f(A))-F_Y(g(A))\rvert.
\]
Furthermore, it is shown in \cite[Proposition 5.1]{Mem_tripods} that, for every dimension $k$,
\begin{equation}\label{eq:stability_and_tripods}
	d_B(\Dgm_k(X,F_X),\Dgm_k(Y,F_Y)))\leq d_{\mathcal F}((X,F_X),(Y,F_Y)).
\end{equation}

Let $(A,X)$ be a metric pair with $X$ finite. For every subset $\sigma$ of $X$, define the {\em circumradius} of $\sigma$ in $A$ as
\begin{equation}\label{eq:circumradius}
    \rad_A(\sigma)=\inf_{a\in A}\max_{x\in \sigma}d(a,x).
\end{equation}

\begin{proposition}\label{prop:tripods}
    Let $(A_1,X_1)$ and $(A_2,X_2)$ be two metric pairs with $X_1$ and $X_2$ finite. Then,
    \[
        d_{\mathcal F}((X_1,\rad_{A_1}),(X_2,\rad_{A_2}))\leq 2\GH((A_1,X_1),(A_2,X_2))
    \]
\end{proposition}
\begin{proof}
    Let $R$ be a correspondence between $A_1$ and $A_2$ such that $R|_{X_1\times X_2}$ is a correspondence. Consider the tripod $(R,\pi_1,\pi_2)$, where $\pi_1$ and $\pi_2$ are the canonical projections restricted to $R$. Let $\sigma$ be a finite subset of $R$. Fix $\varepsilon>0$ and pick $a_1\in A_1$ such that $d_1(a_1,x_1)\leq\rad_{A_1}(\pi_1(\sigma))+\varepsilon$ for every $x_1\in\pi_1(\sigma)$. Arbitrarily choose a point $a_2\in A_2$ such that $(a_1,a_2)\in R$. Then, for every $x_2\in \pi_2(\sigma)$,
    \[
    d_2(x_2,a_2)\leq d_1(x_1,a_1)+\dis R\leq\rad_{A_1}(\pi_1(\sigma))+\dis R+\varepsilon.
    \]
    Therefore, \[\rad_{A_2}(\pi_1(\sigma))\leq\rad_{A_1}(\pi_2(\sigma))+\dis R+\varepsilon.
    \]
    Similarly, we obtain
    \[
    \rad_{A_1}(\pi_1(\sigma))\leq\rad_{A_2}(\pi_2(\sigma))+\dis R+\varepsilon,
    \]
    which shows that $d_{\mathcal F}((X_1,\rad_{A_1}),(X_2,\rad_{A_2}))\leq\dis R+\varepsilon$. Since both $\varepsilon$ and the initial correspondence are arbitrary, the desired inequality follows.
\end{proof}

Given a metric pair $(A,X)$ with $X$ finite, the connection between the filtered space $(X,\rad_A)$ and the ambient \v{C}ech complex $\check{C}(X;A)$ is not immediate in general. However, if the infimum in \eqref{eq:circumradius} is achieved for every $\sigma\subseteq A$, e.g., if $A$ is finite, then the two objects can be identified. In this setting, Corollary~\ref{coro:stability_ambient_cech} immediately follows from Proposition~\ref{prop:tripods} and \eqref{eq:stability_and_tripods}. [[I suspect that even in the general setting the persistence diagrams of $(X,\rad_A)$ and $\check{C}(X;A)$ coincide, don't they?]] In \cite[Proposition 5.2]{Mem_tripods}, the same result was proved for the intrinsic \v{C}ech filtration $\check{C}(X;X)$ and $\check{C}(Y;Y)$ of a pair of finite metric spaces $X$ and $Y$ with respect to their usual Gromov-Hausdorff distance.

\bigskip

\noindent Inria Centre at Universit\'{e} C\^{o}te d'Azur\\
ondrej.draganov@inria.fr

\smallskip 

\noindent University of Vienna\\
sophie.rosenmeier@univie.ac.at

\smallskip

\noindent Institute of Science and Technology Austria (ISTA)\\
nicolo.zava@gmail.com 

\end{document}